\documentclass[11pt,twoside]{article}
\usepackage[a4paper, top=4.5cm, bottom=4.5cm, left=4cm, right=4cm, asymmetric]{geometry}

\usepackage{amsmath,amsthm,amsfonts,amssymb,mathtools,mathdots,scalerel,mathrsfs,stmaryrd,cmll,bbold}
\usepackage[utf8]{inputenc}
\usepackage[T1]{fontenc}
\usepackage{mlmodern,tikz-cd,quiver}
\usepackage[many]{tcolorbox}

\usepackage[shortlabels]{enumitem}
\usepackage[hidelinks]{hyperref}
\usepackage{tikz,fancyhdr,xparse,xcolor,tocloft}
\usepackage[british]{babel}

\usepackage{crossreftools}
\usepackage[textwidth=3cm, textsize=small, colorinlistoftodos]{todonotes}

\usepackage{cctitle,titlesec}
\usepackage{etoolbox}

\usetikzlibrary{nfold}

\makeatletter
\patchcmd{\ttlh@hang}{\parindent\z@}{\parindent\z@\leavevmode}{}{}
\patchcmd{\ttlh@hang}{\noindent}{}{}{}
\makeatother

\newcommand\eqdef{\coloneqq}
\newcommand\nbd{\nobreakdash-\hspace{0pt}}
\newcommand\idd[1]{\mathrm{id}_{#1}}
\newcommand\invrs[1]{#1^{-1}}
\newcommand\after{\circ}

\newcommand\incl{\hookrightarrow}
\newcommand\restr[2]{{#1}{\raisebox{0pt}{$|_{#2}$}}}

\newcommand\set[1]{\left\{ {#1} \right\}}
\newcommand\order[2]{#2^{(#1)}}
\newcommand{\nat}{\mathbb{N}}

\newcommand\cat[1]{\underline{\mathrm{#1}}}
\newcommand\fun[1]{\mathsf{#1}}
\DeclareMathOperator*{\colim}{colim}
\def\raiseslice#1#2{\raisebox{-2pt}{$#1#2$}}
\newcommand{\slice}[2]{{#1{/}\mathpalette\raiseslice{#2}}}
\newcommand{\counit}{\varepsilon}
\renewcommand{\a}{\alpha}

\newcommand\rdcpxequal{\cat{rdCpx}}
\newcommand\rdcpxmap{\cat{rdCpx}^{\downarrow}}

\newcommand\dcpx{\cat{dCpx}}
\newcommand\dcpxomega{\dcpx^{\omega}}
\newcommand\dcpxn[1]{\gr{#1}{\dcpx}}
\newcommand{\dcpxac}{\dcpx^{\mathsf{ac}}}
\newcommand{\Cat}{\cat{Cat}}

\DeclareMathOperator{\Psh}{\cat{Psh}}
\DeclareMathOperator{\sPsh}{\Psh^{\scalebox{0.7}{\( \infty \)}}}
\DeclareMathOperator{\dPsh}{\Psh^{\scalebox{0.7}{\( \Delta \)}}}

\newcommand\Molec{Mol}
\newcommand\Atom{Atom}
\newcommand{\molcat}{\cat{\Molec}}
\newcommand{\atomcat}{\cat{\Atom}}
\newcommand{\atomcatac}{\atomcat^{\mathsf{ac}}
}
\newcommand{\nmolcat}[1]{\gr{#1}{\molcat}}
\newcommand{\natomcat}[1]{\gr{#1}{\atomcat}}

\newcommand\cls[1]{\mathscr{#1}}
\newcommand{\C}{\cls C}

\newcommand\clset[1]{\mathrm{cl}\set{#1}}
\newcommand\gr[2]{#2_{#1}}
\newcommand\maxel[1]{\mathscr{M}\!\mathit{ax}\,#1}
\newcommand\bd[2]{\partial_{#1}^{#2}}
\newcommand\faces[2]{\Delta_{#1}^{#2}}
\newcommand\cofaces[2]{\nabla_{#1}^{#2}}

\newcommand\cp[1]{\,{\scriptstyle\#}_{#1}\,}
\newcommand\celto{\Rightarrow}
\newcommand\submol{\sqsubseteq}
\newcommand\sus[1]{\fun{S}{#1}}
\newcommand\maxflow[2]{\mathscr{M}_{#1}{#2}}
\DeclareMathOperator{\frdim}{frdim}

\newcommand{\join}{\,{\star}\,}
\newcommand{\gray}{\otimes}
\newcommand{\ospx}[1]{\vec{\Delta}^{#1}}

\newcommand{\preLay}[1]{\mathscr{P\!L}\!\mathit{ay}_{#1}}
\newcommand{\preOrd}[1]{\mathscr{P\!O}\!\mathit{rd}_{#1}}
\newcommand{\Lay}[1]{\mathscr{L}\!\mathit{ay}_{#1}}
\newcommand{\Ord}[1]{\mathscr{O}\!\mathit{rd}_{#1}}
\renewcommand{\o}[2]{\mathsf{o}_{#1, #2}}

\newcommand{\mapel}[1]{\iota_{#1}}
\newcommand{\imel}[2]{#1_{#2}}
\newcommand{\prt}[1]{\fun{p}_{#1}}

\newcommand{\globe}[1]{O^{#1}}
\newcommand{\arr}{\vec{I}}
\newcommand{\pt}{\mathbf{1}}

\newcommand{\F}{\fun{F}}
\newcommand{\G}{\fun{G}}

\newcommand{\nCat}[1]{{\mathsf{s}\Cat_{#1}}}
\newcommand{\omegaCat}{{\mathsf{s}\Cat_{\omega}}}
\newcommand{\somegaCat}{{\mathsf{s}\Cat_{\omega}^>}}
\newcommand{\snCat}[1]{{\mathsf{s}\Cat_{#1}^>}}
\newcommand{\nCatinf}[1]{{\Cat_{#1}}}
\newcommand{\nCatflag}[1]{{\Cat_{#1}^{\mathsf{f}}}}

\newcommand{\strweak}{\iota_{\mathsf{s}}}

\newcommand{\skel}[2]{\gr{\le #1}{#2}}
\newcommand{\N}{\fun{N}}
\newcommand{\locflag}{\mathbb{L}^{\mathsf{f}}}

\newcommand{\pcell}[1]{\langle#1\rangle}
\newcommand\molecin[1]{\slice{\mathit{\Molec}}{#1}}

\newcommand{\Pd}{\cat{Pd}\,}
\newcommand{\cell}{\cat{cell}\,}

\newcommand{\ThetaS}[1]{\Theta_{#1}}
\newcommand{\Thetan}[1]{\Theta_{#1}}
\newcommand{\ThetanS}[2]{\Theta_{#1, #2}}
\newcommand{\Rtheta}{\fun{R}}
\newcommand{\Utheta}{\fun{U}}

\newcommand{\rs}{\fun{r}^{>}}

\newcommand{\Sd}{\mathrm{Sd}\,}
\newcommand{\Sdi}{\mathrm{Sd}^\triangleleft\,}
\newcommand{\SdS}[1]{\mathrm{Sd}_{#1}}
\newcommand{\SdiS}[1]{\mathrm{Sd}^\triangleleft_{#1}}
\newcommand{\Sdik}[1]{\SdiS{\set{#1}}}
\newcommand{\Sdk}[1]{\SdS{\set{#1}}}
\newcommand{\bigcell}[1]{\mu_{#1}}
\newcommand{\ivl}[2]{[#1,#2]}

\newcommand{\spine}[1]{\fun{sp}_{#1}}
\newcommand{\Spine}{\fun{Sp}\,}
\newcommand{\Smol}{\fun{Qfr}\,}
\newcommand{\sSmol}{\fun{Qfr}_{\scalebox{0.7}{\( < \)}}\,}
\newcommand{\Smolp}{\fun{Qfr}^+\,}
\newcommand{\sSmolp}{\fun{Qfr}^+_{\scalebox{0.7}{\( < \)}}\,}

\newcommand{\Pol}{\cat{regPol}}
\newcommand{\nPol}[1]{\gr{#1}{\Pol}}
\newcommand{\plexcat}{\cat{Plex}}
\newcommand{\polyplexcat}{\cat{polyPlex}}
\newcommand{\nplexcat}[1]{\gr{#1}{\plexcat}}
\newcommand{\npolyplexcat}[1]{\gr{#1}{\polyplexcat}}
\newcommand{\pl}[1]{\underline{#1}}
\newcommand{\posplex}{\Phi}

\newtheoremstyle{ittheorem}
  {\topsep}   
  {\topsep}   
  {\itshape}  
  {0pt}       
  {\bfseries} 
  { ---}         
  {5pt plus 1pt minus 1pt} 
  {}          

\newtheoremstyle{itdfn}
  {\topsep}
  {\topsep}
  {}
  {0pt}
  {\bfseries}
  {}
  {5pt plus 1pt minus 1pt}
  {\thmnumber{#2}{\thmnote{\normalfont\ \ %
{\sffamily(#3)}.}}}

\newtheoremstyle{itrmk}
  {0.5\topsep}
  {0.5\topsep}
  {\normalfont}
  {0pt}
  {\sffamily \itshape}
  { --- }
  {5pt plus 1pt minus 1pt}
  {}
  
\theoremstyle{ittheorem}
\newtheorem{thm}{Theorem}[section]
\newtheorem*{thm*}{Theorem}
\newtheorem{prop}[thm]{Proposition}
\newtheorem*{prop*}{Proposition}
\newtheorem{cor}[thm]{Corollary}
\newtheorem{lem}[thm]{Lemma}

\newtheorem*{conj*}{Conjecture}
\theoremstyle{itdfn}
\newtheorem{dfn}[thm]{}
\theoremstyle{itrmk}
\newtheorem{rmk}[thm]{Remark}
\newtheorem{comm}[thm]{Comment}
\newtheorem{exm}[thm]{Example}

\setlist{leftmargin=20pt,itemsep=0pt,parsep=0pt,topsep=1ex}

\titleformat{\section}
 {\Large\bfseries}{\thesection}{1em}{}

\titleformat{\subsection}
 {\normalsize\bfseries}{\thesubsection.}{1em}{}

\newcommand\runtitle{a strengthened $(\infty, n)$-categorical pasting theorem}
\newcommand\runauthor{chanavat}

\title{A strengthened $(\infty, n)$-categorical pasting theorem}

\author{Cl\'emence Chanavat}

\institution{Tallinn University of Technology}

\begin{document}

\maketitle
\begin{center}
	\begin{minipage}[t]{.95\textwidth}
		\small\textsc{Abstract.}
		We extend Campion's pasting theorem for $(\infty, n)$\nbd categories to a larger class of polygraphs, called the directed complexes with frame-acyclic molecules.
		It follows, for instance, that this pasting theorem applies to any polygraph presented by a semi-simplicial set, and that a large subclass of directed complexes with frame-acyclic molecules is compatible with the Gray product.  
		We also set up a comparison between directed complexes and Henry's regular polygraphs, and show that they coincide up to dimension $3$.
		As a corollary of our main results, the pasting theorem also applies to the class of regular $3$-polygraphs.
	\end{minipage}
	\vspace{20pt}
	\begin{minipage}[t]{0.95\textwidth}
		\setcounter{tocdepth}{2}
		\tableofcontents
	\end{minipage}
\end{center}

\makeaftertitle

\section*{Introduction}

In higher category theory, a pasting theorem is a tool guaranteeing that some freely-generated categories in a strict setting are also freely-generated in a homotopical sense.
The simplest example is that of a graph: is it the case that the free category on a graph is also its free \( \infty \)\nbd category?
The answer to that question is well-known to be yes; this is the 1-dimensional pasting theorem, which is, in some sense, as strong as one could hope.

To even state a pasting theorem in higher dimensions, we must first specify what we mean by ``freely-generated higher category''.
The most general formalism is that of \emph{polygraphs} \cite{burroni1993higher}, also known as \emph{computads}.
An \( n \)\nbd polygraph is, roughly, a strict \( n \)\nbd category\footnote{In this article, we do not distinguish between a polygraph and the strict \( \omega \)\nbd category it generates.} which is ``level-wise'' freely generated: a \( 0 \)\nbd polygraph is a set, and inductively, an \( n \)\nbd polygraph is built out \emph{cellular extensions} of an \( (n -1) \)\nbd polygraph, whereby cellular extensions, we mean that we freely attach a collection of \( n \)\nbd dimensional cells at some prescribed \( (n - 1) \)\nbd dimensional boundaries.
This attaching operation is realised via a pushout in the category of strict \( n \)\nbd categories. 
The pasting theorem then gives conditions under which this pushout is also a homotopy pushout when passing to \( (\infty, n) \)\nbd categories, allowing to strictify homotopy-coherent data.

One particularly interesting application concerns the theory of \emph{pasting diagrams}, which are formalisms providing generalisations of commutative diagrams to higher dimensions.
A key feature of pasting diagrams, or rather, of the polygraph they generate, is the uniqueness of the composite: the generating cells of a pasting diagram assemble uniquely into a ``big cell''; at least in a strict \( n \)\nbd category, see \cite{johnson1989combinatorics, power1991pasting} for strict \( n \)\nbd categorical pasting theorems.
In the homotopical setting, a pasting theorem then proves that the composite is unique up to a contractible space of choices.

The literature contains many results guaranteeing that certain pushout diagrams are homotopical, to name a few: \cite{hackney2022pasting} and \cite{columbus20182categorical} provide a pasting theorem for pasting diagrams in an \( (\infty, 2) \)\nbd category, while \cite{ozornova2022pushout} implies an \( (\infty, n) \)\nbd pasting theorem for the thetas in the complicial model of \( (\infty, n) \)\nbd categories \cite{verity2006weak}.
In a different direction, \cite{ruit2024pasting} gives a pasting theorem for iterated Segal spaces, which are the homotopical version of \( n \)\nbd uple categories.
Finally, the most general pasting theorem currently available for \( (\infty, n) \)\nbd categories is that of Campion \cite{campion2023pasting}, using Forest's \cite{forest2022pasting} torsion-free complexes as underlying formalism of pasting diagrams.

In this article, we extend the reach of Campion's result in two ways:
\begin{enumerate}
    \item the first generalisation concerns the pasting diagrams. 
    The underlying formalism is now given by Hadzihasanovic's regular directed complexes \cite{hadzihasanovic2024combinatorics}, which is a theory of higher categorical diagrams beyond acyclicity.
    While the pasting diagrams of this formalism, called the \emph{molecules}, still generate a strict \( \omega \)\nbd category with a unique composite, something peculiar happens: these strict \( \omega \)\nbd categories are, in general, not polygraphs. 
    This fact can be understood by a change of precedence.
    In all previous formalisms, the axioms of strict \( \omega \)\nbd categories came first: associativity, exchange, and unitality.
    The goal was then to analytically describe classes of combinatorial structures that would generate pasting diagrams.   
    Conversely, Hadzihasanovic takes a synthetic perspective: the combinatorial pasting diagrams are inductively generated by topologically sound operations on combinatorial cell complexes, where, by design, they have a unique composite.
    It is now the responsibility of the algebraic structures to adapt their axioms so that the uniqueness of this composite is still freely generated. 
    As it turns out, these axioms indeed include exchange and associativity, but starting dimension \( 4 \), they are more than that; we refer the reader to \cite{chanavat2025stricter} for a theory of \emph{stricter} \( n \)\nbd categories, that is, of strict \( n \)\nbd categories satisfying those extra axioms. 
    Nonetheless, following insights of Steiner \cite{street1987algebra}, Hadzihasanovic introduced a condition called \emph{frame-acyclicity} (Definition \ref{dfn:frame_acyclic}), which is a general condition guaranteeing that a molecule generates a polygraph.
    For instance, all molecules of dimension \( \le 3 \) are frame-acyclic, and all acyclic molecules are in particular frame-acyclic.
    Our pasting theorem will then generally apply to frame-acyclic molecules.

    \item the second generalisation, and the most important for the applications, is the relative approach permitted by the functorial viewpoint developed in \cite{hadzihasanovic2024combinatorics}: what matters is not that the directed complex satisfies a certain acyclicity condition, but merely that the pasting diagrams mapping to it satisfy the condition.
    For instance, consider the following graph
    
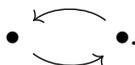
\begin{center}
        \begin{tikzcd}
            \bullet & \bullet.
            \arrow[curve={height=12pt}, from=1-1, to=1-2]
            \arrow[curve={height=12pt}, from=1-2, to=1-1]
        \end{tikzcd}
    \end{center}  
    Although extremely simple, it is excluded from Campion's pasting theorem because there is a cycle in its generating cells.
    Yet, the pasting diagrams mapping into it \emph{are} acyclic, for the reason that all the \( 1 \)\nbd dimensional pasting diagrams are acyclic (they are given by the simplices seen as free categories on chains of arrows).
    Since the above graph \emph{has acyclic pasting diagrams}, as opposed to \emph{is acyclic}, our pasting theorem applies to it. 
\end{enumerate} 

\noindent The most general form of our main result (Theorem \ref{thm:frame_acyclic_is_homotopy_polygraph}) is as follows:
\begin{thm*}
    Let \( X \) be a directed complex with frame-acyclic molecules.
    Then \( \molecin{X} \) is a polygraph and a homotopy polygraph.
\end{thm*}

\noindent Here, \( \molecin{X} \) is the strict \( \omega \)\nbd category generated by any directed complex; its cells are given by the pasting diagrams in \( X \), which are required to be frame-acyclic.
By homotopy polygraph, we mean that the cellular extensions computed in strict \( \omega \)\nbd categories are preserved by the inclusion \( \strweak \) of the latter into \emph{flagged \( (\infty, \omega) \)\nbd categories} defined in the sense of \cite{ayala2018flagged} as the localisation of space-valued presheaves over the theta category at the basic spine inclusions of thetas. 

Since the category of directed complexes with frame-acyclic molecules has the category of semi-simplicial sets as a full subcategory, we obtain immediately (Theorem \ref{thm:simplicial_homotopy_polygraph}):
\begin{thm*}
    Let \( X \) be a semi-simplicial set.
    Then \( \molecin{X} \) is a homotopy polygraph.
\end{thm*}

\noindent Notice that in that case, the operation \( \molecin{-} \) sends simplices to orientals \cite{street1987algebra}, thus \( \molecin{X} \) is the polygraph generated by a semi-simplicial set in the usual way.  
As far as we are aware, this is a new result. 
In fact, we may replace \( X \) with any \emph{directed complex with acyclic atoms} (Definition \ref{dfn:acyclic_atoms}) and still obtain a homotopy polygraph.
While having acyclic atoms is strictly stronger than having frame-acyclic molecules, it is a condition easier to check in practice, since it concerns only the cells of a directed complex, as opposed to all of its pasting diagrams.
Furthermore, directed complexes with acyclic atoms form a presheaf category \( \dcpxac \), which support a monoidal structure given by the Gray tensor product \( - \gray - \).
Another application of our main result is Theorem \ref{thm:gray_monoidal_acyclic}, which gives a large class of polygraphs for which it is possible to compute the (homotopical) Gray tensor product combinatorially at the level of the underlying complexes: 
\begin{thm*}
    The functor \( \strweak\molecin{-} \colon \dcpxac \to \nCatflag{\omega} \) lifts to a (strong) monoidal functor
    \begin{equation*}
        \strweak\molecin{-} \colon (\dcpxac, \gray, \pt) \to (\nCatflag{\omega}, \gray, \pt).
    \end{equation*}
\end{thm*}

\noindent Finally, we stated earlier that all molecules of dimension \( \le 3 \) are frame-acyclic.
Thus, the pasting theorem applies in particular to directed \( 3 \)\nbd complexes (Theorem \ref{thm:dim_le_3_frame_acyclic}).
The reader familiar with the work of Henry may know of similarities between his regular polygraphs \cite{henry2018regular} and directed complexes. 
It was conjectured by Hadzihasanovic and Henry that directed complexes and regular polygraphs should agree up to dimension \( 3 \).
We take this opportunity to finally set up a comparison in Theorem \ref{thm:from_polyplexes_to_molecules}, reproduced here:

\begin{thm*}
    Let \( n \in \nat \).
    Then the diagram
    \begin{center}
        \begin{tikzcd}
            {\nplexcat{n}} & {\natomcat n} \\
            {\nPol n} & {\dcpxn n} \\
            {\nCat n} & {\snCat n}
            \arrow["\posplex", from=1-1, to=1-2]
            \arrow[hook, from=1-1, to=2-1]
            \arrow[hook, from=1-2, to=2-2]
            \arrow["\posplex", from=2-1, to=2-2]
            \arrow[hook, from=2-1, to=3-1]
            \arrow["{\molecin{-}}", hook, from=2-2, to=3-2]
            \arrow["\rs", from=3-1, to=3-2]
        \end{tikzcd}
    \end{center}
    commutes, and the horizontal functors are equivalences if \( n \le 3 \).
\end{thm*}

\noindent Here, \( \nplexcat{n} \) is the category of \emph{regular plexes} (of dimension \( \le n \)), Henry's notion of (regular) pasting diagrams with a greatest generating cell, and \( \natomcat{n} \) is the category of atoms (of dimension \( \le n \)), Hadzihasanovic's notion of pasting diagrams with a greatest generating cell.
The functor \( \posplex \colon \nPol n \to \dcpxn{n} \) sends a regular polygraph to its associated \emph{stricter} regular polygraph, in a way that is compatible with the reflector \( \rs \) from strict \( n \)\nbd categories to stricter \( n \)\nbd categories.
As an immediate consequence, we get Corollary \ref{cor:3regular_polygraph_are_homotopy_polygraph}, which is, as far as we are aware, a new result as well:

\begin{thm*}
    All regular \( 3 \)\nbd polygraphs are homotopy polygraphs.
\end{thm*}

\subsubsection*{Flagged \texorpdfstring{$(\infty, n)$}{(infty, n)}-categories}

Following \cite{campion2023pasting}, we decided to state the pasting theorem in the context of \emph{flagged \( (\infty, n) \)\nbd categories } \cite{ayala2018flagged}, which are the ``non-univalent \( (\infty, n) \)\nbd categories''.
Since the usual \( \infty \)\nbd category of \( (\infty, n) \)\nbd categories \( \nCatinf{n} \) is a localisation of the \( \infty \)\nbd category of flagged \( (\infty, n) \)\nbd categories, and localisation preserves homotopy pushouts, the pasting theorem directly applies to them.
This is particularly convenient in the case \( n = \omega \), since ``the'' model of \( (\infty, \omega) \)\nbd categories is not well-defined, as there seems to be several reasonable, yet non-equivalent, possible choices \cite{ozornova2026cores}.

\subsection*{Towards a comparison with the diagrammatic model}

In recent years, together with Hadzihasanovic, we developed the \emph{diagrammatic model of \( (\infty, n) \)\nbd categories} \cite{chanavat2024model}.
The motivating reason for this article was to make progress toward comparing the \( \infty \)\nbd category of diagrammatic \( (\infty, n) \)\nbd categories with \( \nCatinf{n} \).
In \cite{chanavat2025stricter}, we already settled the \( (0, n) \)\nbd case: the ``correct'' notion of strict \( n \)\nbd category for the diagrammatic model is that of \emph{stricter} \( n \)\nbd category.
We showed that stricter \( n \)\nbd categories are a reflective localisation of strict \( n \)\nbd category at the fundamental spines of molecules (Definition \ref{dfn:spine}), and that this localisation is trivial if \( n \le 3 \).
We now fully believe that the same is true in the homotopical setting: the \( \infty \)\nbd category of diagrammatic \( (\infty, n) \)\nbd categories should be a reflective localisation of the \( \infty \)\nbd category of \( (\infty, n) \)\nbd categories.
More precisely,
\begin{conj*}
    The \( \infty \)\nbd category of diagrammatic \( (\infty, n) \)\nbd categories is the localisation of \( \nCatinf{n} \) at the set of fundamental spines of \( n \)\nbd molecules.
\end{conj*} 

\noindent From Corollary \ref{cor:spine_of_3_rdcpx_is_flagged_equivalence}, we get

\begin{thm*}
    Let \( U \) be a molecule of dimension \( \le 3 \). 
    Then the fundamental spine of \( U \) is an equivalence.
\end{thm*}

\noindent Provided the previous conjecture holds, this shows that, akin to the strict case, the localisation is trivial when \( n \le 3 \).
In the sequel, we will show how to make progress towards this conjecture.

\subsection*{Related work}

The main results of this article owe mostly to the insights of Campion in \cite{campion2023pasting}.
While we were able to reuse his strategy to reach our main theorem with very little divergence, we felt it would have been too hasty to simply affirm, in later work, that Campion's proof straightforwardly adapts to directed complexes with frame-acyclic molecules: the core of the proof is indeed a straightforward adaptation thereof, but the formalism surrounding it is fundamentally different, and, since our main results substantially extend the current range of application of Campion's theorem, we could not be completely sure that we did not miss some crucial details, if not by writing this article.

For the reader familiar with \cite{campion2023pasting}, we will give, throughout the article, detailed similarities and differences between our work and his.
In particular, we always mention when a definition or a proof is adapted from \cite{campion2023pasting}, and allowed ourselves to skip some proofs when they were the same.

\subsection*{Organisation of the article}

In Section \ref{sec:rdcpx}, we first review the basic theory of directed complexes, which is the general notion of positive and regular ``stricter'' polygraph.
Of particular interest is Proposition \ref{prop:ofs_functor_of_dcpx} stating that strict functors of directed complexes have an ``active-inert'' orthogonal factorisation system, which generalises the usual one on the simplex category. 
We conclude the section by the comparison between regular polygraphs and directed complexes in Theorem \ref{thm:from_polyplexes_to_molecules}.

Next, we start Section \ref{sec:subdiv} with a slight extension of Hadzihasanovic's theory of layerings and orderings into what we call pre-layerings and pre-orderings (Definition \ref{dfn:prelayering} and Definition \ref{dfn:preordering}).
While a \( k \)\nbd layering of a molecule, when it exists, is a way to maximally decompose it into pastings along the \( k \)\nbd boundary, a \( k \)\nbd pre-layering is a similar decomposition without the maximality condition.
Following Campion's idea, we organise the collection of \( k \)\nbd pre-layerings into a poset, and the key result of that section asserts that whenever a molecule is frame-acyclic, this poset is isomorphic to the poset of \( k \)\nbd pre-orderings (Proposition \ref{prop:frame_acyclic_iso_preorderings}).
In turn, the poset of \( k \)\nbd pre-orderings has a very combinatorial definition: it is the poset of linearly ordered partitions refining the \emph{maximal \( k \)\nbd flow graph}, a graph associated to a molecule (Definition \ref{dfn:maxflow}); the acyclicity of this graph at a well-chosen \( k \) is key to the condition of frame-acyclicity.

After that, we give some recollections on the theta category \( \Theta \), the key technical point is Proposition \ref{prop:quotient-free_non_degenerate_are_mono}, which characterises the non-degenerate sections \( u \colon \theta \to \molecin{U} \), when \( U \) is a molecule (or more generally a regular directed complex).
In Campion's proof, an important requirement states that any non-degenerate section into the polygraph presented by a torsion-free complex is a monomorphism.
This is no longer the case when \( U \) is a molecule, but we overcome the issue by noticing that this requirement only really matters at later stages when \( u \) is an active (or more generally \emph{quotient-free}) strict functor, see also Comment \ref{comm:alternative_non_degenerate}.

Finally, following once again Campion's insight, we are quickly able to deduce that the \emph{poset of subdivisions} of a molecule (Definition \ref{dfn:poset_subdiv}) is contractible whenever the latter is frame-acyclic.
This poset encodes all the possible ways to decompose a molecule into a pasting of submolecules, in a way that follows the decomposition of some theta into the pasting of its globes.  
Formally, if \( U \) is a molecule, the \emph{initial poset of subdivisions} \( \Sdi U \) is defined as the full subcategory of the category of elements of \( \molecin{U} \) on non-degenerate active strict functors. 
By the previous paragraph, this is a poset of subobjects.
The poset of subdivision is then given by \( \Sdi U \), its initial object removed.
To conclude it is contractible in the frame-acyclic case, we will appeal to Campion's proof of that same result (\cite[Theorem 4.21]{campion2023pasting}), which goes through essentially unmodified.
There are only two hypothesis to check.
The first one requires that the poset of \( k \)\nbd pre-layering (for a well-chosen \( k \)) is contractible; this already follows from the previous isomorphisms between pre-layerings and pre-orderings, and \cite[Corollary 3.4]{campion2023pasting}.
We must finally show that certain functors between subposets of \( \Sdi U \) are cocartesian fibrations, since Campion's proof of contractibility relies on Quillen's theorem A, which is easier to apply in that case.
Our proof that those functors are cocartesian is given by Lemma \ref{lem:compose_is_cocartesian_one_step}, and essentially matches the one of Campion. 
We decided to include it for reasons we explain in Comment \ref{comm:why_proof_cocart}.
We conclude the section by Theorem \ref{thm:poset_of_subdivision_contractible_or_empty}, asserting that \( \Sd U \) is contractible (or empty if the molecule is an atom).
As it will made precise in the next section, this means that the (flagged) \( (\infty, n) \)\nbd category generated by a frame-acyclic molecule is a pasting diagram: it has a unique, up to a contractible space of choices, composite.  

Last is Section \ref{sec:pasting}, which proves the pasting theorem.
We give a short recollection of flagged \( (\infty, n) \)\nbd categories, specify what we mean for a polygraph to be a homotopy polygraph (Definition \ref{dfn:htpy_polygraph}) in that context, and proceed to the proof of the main result.
One more time, we use the result of \cite[Section 5]{campion2023pasting}, which gets us to Proposition \ref{prop:frame_acyclic_can_add_active}, asserting that we may extend a certain subobject of \( \molecin{U} \), when \( U \) is a frame-acyclic molecule, into a larger one including all of the active functors.
We may then glue together these equivalences in Theorem \ref{thm:cellular_extension_of_frame_acyclic}, the analogue of \cite[Theorem 5.11]{campion2023pasting}.
From there, we conclude with the direct corollaries mentioned earlier in the introduction.   

\subsection*{Acknowledgments}
 
This work is the first of a series of two articles\footnote{The sequel will concern advances toward the comparison between the diagrammatic and the standard models of \( (\infty, n) \)\nbd categories.} that elaborate on results and conjectures announced during our talk at the MPIM last February.
We thank Viktoriya Ozornova and L\'eo Schelstraete for the invitation and subsequent interesting discussions. 
We are grateful to Amar Hadzihasanovic for the precious feedback.
We thank Marcus Nicolas for suggesting to us Theorem \ref{thm:gray_monoidal_acyclic} concerning the Gray tensor product, and Tim Campion for pointing out a reference and useful discussions concerning an earlier version of this article. 

\section{Directed complexes and their strict functors}
\label{sec:rdcpx}

\subsection{Molecules and directed complexes}

\noindent We refer the reader to \cite{hadzihasanovic2024combinatorics} for more details on the following definitions.

\begin{dfn}
    An oriented graded poset is a graded poset \( P \) together with a partition
    \begin{equation*}
        \faces{}{} x = \faces{}{-} x + \faces{}{+} x
    \end{equation*}
    of the faces of any \( x \in P \) into inputs faces and output faces.
    A \emph{morphism \( f \colon P \to Q \) of oriented graded posets} is a function of their underlying sets that induces, for all \( x \in P \) and \( \a \in \set{-, +} \), a bijection between \( \faces{}{\a} x \) and \( \faces{}{\a} f(x) \).
    We say that \( f \) is an \emph{embedding} if it is injective.
    The grading of \( P \) gives, for each \( x \in P \), the \emph{dimension of \( x \)}, written \( \dim x \).
    We write \( \dim P \) for the maximum of the dimension of the elements of \( P \), or \( -1 \) if \( P \) is empty.
    For any subset \( A \subseteq P \) of \( P \), and \( k \in \nat \), we write \( \gr{k}{A} \) for the set of elements of \( A \) whose dimension is \( k \).
    We write \( \maxel{P} \) for the set of maximal elements of \( P \).
    For each \( k \in \nat \) and \( \a \in \set{-, +} \), an oriented graded poset has an intrinsic notion of input and output \( k \)\nbd boundaries, written \( \bd{k}{\a} P \).
    They are closed subsets of \( P \) of dimension \( \min{\set{\dim P, k}} \). 
    The convention is such that \( \bd{k}{\a} P \) is empty when \( k < 0 \).  
\end{dfn}

\begin{dfn} 
    The class of \emph{molecules} is an inductively generated class of oriented graded posets.
    Given two molecules \( U, V \), and \( k \in \nat \), if there is a (necessarily unique) isomorphism \( \bd{k}{+} U \cong \bd{k}{-} V \) in the category of oriented graded posets, then the pushout of \( U \) and \( V \) along their common boundary is a molecule, written \( U \cp{k} V \), and called the \emph{pasting of \( U \) and \( V \) along the \( k \)\nbd boundary}. 
    Boundaries and pasting make the collection of molecules a strict \( \omega \)\nbd category. 
    An \emph{atom} is a molecule with a greatest element.
    A \emph{regular directed complex} is an oriented graded poset such that for all \( x \in P \), \( \clset{x} \) is an atom, where \( \clset{x} \) denotes the downward closure of \( \set{x} \).
    A molecule \( U \) is \emph{round} if for all \( k < \dim U \), the closed subsets 
    \begin{equation*}
        \bd{k - 1}{} U \eqdef \bd{k - 1}{-} U \cup \bd{k - 1}{+} U \quad\text{ and }\quad \bd{k}{-} U \cap \bd{k}{+} U
    \end{equation*}
    are equal.
    Given two round molecules \( U \) and \( V \) of the same dimension \( k \geq 0 \), if \( \bd{k - 1}{\a} U \cong \bd{k - 1}{\a} V \) for all \( \a \in \set{-, +} \), we write \( U \celto V \) for the unique \( (k + 1) \)\nbd dimensional atom such that \( \bd{k}{-} (U \celto V) = U \) and \( \bd{k}{+} (U \celto V) = V \).
\end{dfn}

\begin{exm} 
    The \emph{point} is the atom \( \pt \), whose underlying graded poset has a unique element.
    The \emph{arrow} is the atom \( \arr \eqdef \pt \celto \pt \).
    Given \( k \geq 0 \), we write
    \begin{equation*}
        k\arr \eqdef \underbrace{\arr \cp{0} \ldots \cp{0} \arr}_{k \text{ times}}.
    \end{equation*}
    When \( k = 0 \), \( k\arr \) is understood to be \( \pt \).
\end{exm}

\begin{dfn} [Final map]
    We say that an order-preserving map of poset is \emph{final}, when it is final when seen as a functor between posets seen as categories.
\end{dfn}

\begin{dfn} [Map of regular directed complexes]
    Let \( P \) and \( Q \) be regular directed complexes.
    A \emph{map} \( f \colon P \to Q \) is a closed order-preserving function of their underlying posets such that for all \( x \in P \), \( n \in \nat \), and \( \a \in \set{-, +} \), \( f(\bd{n}{\a} x) = \bd{n}{\a} f(x) \) and the restriction
    \begin{equation*}
        \restr{f}{\bd{n}{\a} x} \colon \bd{n}{\a} x \to f(\bd{n}{\a} x)
    \end{equation*}
    is final.
    We write \( \rdcpxmap \) for the category of regular directed complexes and maps. 
\end{dfn}

\begin{dfn}
    Let \( f \colon P \to Q \) be a map of regular directed complexes.
    We say that \( f \) is
    \begin{itemize}
        \item a \emph{collapse} if \( f \) is final;
        \item an \emph{embedding} if \( f \) is injective;
        \item a \emph{local embedding} if for all \( x \in P \), \( \restr{f}{\clset{x}} \) is an embedding.
    \end{itemize}
    We write \( \rdcpxequal \) for the wide subcategory of \( \rdcpxmap \) on local embeddings.
    We also let \( \molcat \) and \( \atomcat \) be the full subcategories of \( \rdcpxequal \) on molecules and atoms respectively.
\end{dfn}

\begin{rmk}
    By \cite[Proposition 6.2.19]{hadzihasanovic2024combinatorics} the category \( \rdcpxequal \) is also the full subcategory of the category of oriented graded posets and their morphisms on regular directed complexes.
\end{rmk}

\begin{rmk} \label{rmk:factorisation_collapse_local_embedding}
    By \cite[Corollary 6.2.31]{hadzihasanovic2024combinatorics}, collapses and local embeddings form an orthogonal factorisation system on \( \rdcpxmap \).
\end{rmk}

\begin{exm}
    Each closed subset \( A \subseteq P \) of a regular directed complex determines an embedding \( \iota \colon A \incl P \).
    In particular, if \( A = \clset{x} \) for \( x \in P \), we write \( \mapel{x} \colon \imel{P}{x} \incl P \) for this canonical embedding. 
\end{exm}

\begin{dfn} [Submolecules inclusion]
    The \emph{class of submolecule inclusions} is the smallest class of embeddings of molecules that are closed under composition, contain all isomorphisms, and all embeddings of the form
    \begin{equation*}
        U \incl U \cp{k} V,\quad \text{ and }\quad V \incl U \cp{k} V
    \end{equation*}
    whenever a pasting \( U \cp{k} V \) is defined.
    If \( U \) is a molecule and \( V \subseteq U \) is a closed subset, we say that \emph{\( V \) is a submolecule of \( U \)}, written \( V \submol U \), if the embedding of \( V \) into \( U \) is a submolecule inclusion.
\end{dfn}

\begin{exm}
    Any embedding of molecules whose domain is an atom is a submolecule inclusion.
    Any embedding of the form \( \bd{n}{\a} U \incl U \) is a submolecule inclusion.
\end{exm}

\begin{comm} \label{comm:induction_submolecules}
    Since the relation \( \submol \) is well-founded, we may prove statements on molecules by \emph{induction on submolecules}, see also [Had24, Comment 4.1.7].
\end{comm}

\begin{dfn} [Split of molecule]
    Let \( U \) be a molecule.
    We say that \( U \) \emph{splits} as \( A \cp{k} B \) if \( U \) is isomorphic to \( A \cp{k} B \), for \( A, B \) molecules such that \( A \cp{k} B \) is defined.
\end{dfn}

\begin{dfn} [Directed complex]
    A \emph{directed complex} is a presheaf on \( \atomcat \).
    We write \( \dcpx \) for the category of directed complexes and their natural transformations as morphisms. 
\end{dfn}

\noindent By a simple variant of \cite[Proposition 2.11]{chanavat2026htpy}, we may identify the category \( \rdcpxequal \) with the full subcategory of \( \dcpx \) on the directed complexes \( X \) having the property that for all atoms \( U \), any section \( x \colon U \to X \) is a monomorphism. 
Then the full subcategory inclusion 
\begin{equation*}
    \rdcpxequal \incl \dcpx
\end{equation*}
factors through the Yoneda embedding and preserves colimits of embeddings (when they exist in the domain).

\begin{dfn} [Embedding of directed complexes]
    We say that a morphism of directed complexes is an \emph{embedding} if it is a monomorphism; this terminology is consistent with the identification of \( \rdcpxequal \) as a full subcategory of \( \dcpx \).
\end{dfn}

\begin{dfn} [Skeleton of a directed complex] \label{dfn:skeleton_dcpx}
    Let \( X \) be a directed complex, and \( n \in \nat \cup \set{-1, \omega} \).
    We write \( \gr{n}{X} \) for the collection of sections \( x \colon U \to X \) with \( U \) an atom of dimension \( n \).
    The \emph{\( n \)\nbd skeleton of \( X \)} is the directed complex \( \skel{n}{X} \) defined by
    \begin{equation*}
        U \mapsto 
        \begin{cases}
            \varnothing & \text{if } \dim U > n, \\
            X(U) & \text{else}.
        \end{cases}
    \end{equation*}
    This is a subobject of \( X \), and we have a canonical embedding \( \skel{n}{X} \incl \skel{n + 1}{X} \) fitting in a pushout diagram
    \begin{center}
        \begin{tikzcd}
            {\coprod\limits_{x \in \gr {n + 1} X} \bd{}{}U} & {\coprod\limits_{x \in \gr {n + 1} X} U} \\
            {\skel{n}{X}} & {\skel{n+1}{X}.}
            \arrow[""{name=0, anchor=center, inner sep=0}, hook, from=1-1, to=1-2]
            \arrow[from=1-1, to=2-1]
            \arrow[from=1-2, to=2-2]
            \arrow[hook, from=2-1, to=2-2]
            \arrow["\lrcorner"{anchor=center, pos=0.125, rotate=180}, draw=none, from=2-2, to=0]
        \end{tikzcd}
    \end{center}
    Finally, we have \( X = \bigcup_{n \geq 0} \skel{n}{X} \).
\end{dfn}

\begin{rmk}
    By definition, \( \skel{-1}{X} = \varnothing \) and \( \skel{\omega}{X} = X \).
\end{rmk}

\begin{dfn} [Dimension of a directed complex]
    Let \( X \) be a directed complex.
    Then \emph{dimension of \( X \)}, written \( \dim X \), is the smallest element \( d \in \nat \cup \set{-1, \omega} \) such that \( \skel{d}{X} = X \). 
\end{dfn}

\begin{rmk}
    If \( X \) is a regular directed complex, then \( \dim X \) is also the dimension of its underlying graded poset. 
\end{rmk}

\begin{dfn} [Directed \( n \)\nbd complex]
    Let \( n \in \nat \cup \set{\omega} \).
    A \emph{directed \( n \)\nbd complex} is a directed complex \( X \) such that \( \dim X \le n \).
    We write \( \dcpxn n \) for the full subcategory of \( \dcpx \) on directed \( n \)\nbd complexes.
    If \( X \) is a regular directed complex or a molecule, we speak of regular directed \( n \)\nbd complex and \( n \)\nbd molecule.
    We also write \( \nmolcat{n} \) and \( \natomcat{n} \) for the subcategories of \( \dcpxn{n} \) on molecules and atoms respectively.
\end{dfn}

\begin{rmk}
    The category \( \dcpxn{n} \) is also the category of presheaves on the subcategory of \( \atomcat \) on \( n \)\nbd atoms, and the inclusion \( \iota \colon \dcpxn{n} \incl \dcpx \) has a right adjoint \( \skel{n}{(-)} \colon \dcpx \to \dcpxn{n} \) that discards cells of dimension \( > n \).
\end{rmk}

\begin{dfn} [Diagrams in a directed complex]
    Let \( X \) be a directed complex.
    A \emph{diagram of shape \( U \) in \( X \)} is a morphism \( f \colon U \to X \) where \( U \) is a directed complex.
    A diagram is called:
    \begin{itemize}
        \item a \emph{pasting diagram} if \( U \) is a molecule, and
        \item a \emph{cell} if \( U \) is an atom.
    \end{itemize}
    The \emph{dimension of \( f \)} is \( \dim f \eqdef \dim U \), which belongs to \( \nat \cup \set{-1, \omega} \).  
    We write \( \Pd X \) and \( \cell X \) for the full subcategories of the slice category \( \slice{\dcpx}{X} \) on pasting diagrams and cells respectively.
\end{dfn}

\begin{rmk}
    By the Yoneda Lemma, \( \cell X \) is the category of elements of \( X \).
\end{rmk}

\begin{dfn} [Subdiagram]
    Let \( X \) be a directed complex and \( f \colon U \to X \) be a pasting diagram.
    A \emph{subdiagram of \( f \)} is the data of a pasting diagram \( g \colon V \to X \) together with a submolecule \( \iota \colon V \incl U \) such that \( f \after \iota = g \).
    We write \( \iota \colon g \submol f \) for the data of a subdiagram, or simply \( g \submol f \) if \( \iota \) is irrelevant or evident from the context. 
\end{dfn}

\begin{dfn} [Pasting and boundaries of pasting diagrams]
    Let \( X \) be a directed complex, and \( f \colon U \to X \) be a pasting diagram.
    For each \( \a \in \set{-, +} \) and \( k \in \nat \), we write \( \bd{k}{\a} f \colon \bd{k}{\a} U \to X \) for the subdiagram of \( f \) obtained by precomposing \( f \) with \( \bd{k}{\a} U \incl U \).
    If \( g \colon V \to X \) is another pasting diagram in \( X \), and \( \bd{k}{+} f = \bd{k}{-} g \), we write \( f \cp{k} g \colon U \cp{k} V \to X \) for the pasting diagram determined by the universal property of the pushout.
\end{dfn}

\subsection{Strict functors of directed complexes}

\noindent Recall that a (small) strict \( \omega \)\nbd category is a set \( C \) together with, for all \( k \in \nat \), \( \a \in \set{-, +} \), globular boundary operators \( \bd{k}{\a} \colon C \to C \), and for all \( k \geq 0 \), a \( k \)\nbd composition operation satisfying a certain list of axioms including associativity, exchange, and unitality axioms.
In particular, each \( x \in C \) has a dimension, defined as the minimal \( k \in \nat \) such that \( \bd{k}{+} x = x = \bd{k}{-} x \), and the elements of dimension \( k \) will be called the \emph{\( k \)\nbd arrows}\footnote{We decided to use the word arrow in order to avoid an extra overloading of the term cell.} of \( C \).
If \( d \in \nat \), we write \( \globe{d} \) for the \( d \)\nbd globe: a strict functor \( \F \colon \globe{d} \to C \) classifies a unique arrow of dimension \( \le d \).

We write \( \omegaCat \) for the category of strict \( \omega \)\nbd categories and strict functors, and, given \( n \in \nat \cup \set{\omega} \), we let \( \nCat{n} \) be its full subcategory on strict \( \omega \)\nbd categories whose underlying globular set is \( n \)\nbd skeletal, that is, there are no \( k \)\nbd arrows, for \( k > n \).
The inclusion \( \nCat{n} \incl \omegaCat \) has a right adjoint \( \skel{n}{(-)} \colon \omegaCat \to \nCat{n} \) given by discarding arrows of dimension \( > n \).

\begin{dfn} [The strict \( \omega \)\nbd category of molecules over a directed complex]
    Let \( X \) be a directed complex.
    The \emph{strict \( \omega \)\nbd category of molecules over \( X \)} is the strict \( \omega \)\nbd category \( \molecin{X} \) whose
    \begin{itemize}
        \item \( k \)\nbd arrows are pasting diagrams \( f \colon U \to X \) with \( \dim f = k \);
        \item boundary operators are given by \( f \mapsto \bd{k}{\a} f \);
        \item composition operation is given by \( (f, g) \mapsto f \cp{k} g \). 
    \end{itemize}
    We call \( \dcpxomega \) the full subcategory of strict \( \omega \)\nbd categories on objects \( \molecin{X} \), for \( X \) a directed complex.
    The assignment \( X \mapsto \molecin{X} \) extends to a functor \( \molecin{-} \colon \dcpx \to \omegaCat \).
\end{dfn}

\begin{rmk} \label{rmk:atom_basis_mol}
    If \( X \) is a directed \( n \)\nbd complex, then by a fibered variant of \cite[Theorem 5.2.5 and Proposition 5.2.7]{hadzihasanovic2024combinatorics}, \( \molecin{X} \) is indeed a strict \( n \)\nbd category.
    By the same result, the objects of \( \cell{X} \) are a basis for \( \molecin{X} \).
    In particular, two parallel strict functors \( \F \) and \( \G \) from \( \molecin{X} \) are equal if and only if, for all cells \( x \colon U \to X \), \( \F(x) = \G(x) \)
\end{rmk}

\begin{rmk}
    When \( X \) is a regular directed complex, then the pasting diagrams in \( X \) are precisely the local embeddings \( f \colon U \to X \) with \( U \) a molecule, hence \( \molecin{X} \) coincides with \cite[Definition 5.2.8]{hadzihasanovic2024combinatorics}.
\end{rmk}

\begin{comm}
    To be completely formal, the arrows of \( \molecin{X} \) should be equivalence classes of local embeddings \( [f \colon U \to P] \), otherwise the axioms of strict \( \omega \)\nbd category are only defined up to isomorphism.
    However, since isomorphisms of molecules are unique when they exist, we may safely forget about this. 
\end{comm}

\begin{exm}
    For all \( k \geq 0 \), \( \molecin{k\arr} \) is the category isomorphic to the poset \( \set{0 < \ldots < k} \).
    More generally, a graph \( X \) is just a directed \( 1 \)\nbd complex and \( \molecin{X} \) is the free category on it.  
    Going one dimension up, \( \molecin{(\arr \celto 2\arr)} \) is Street's second oriental \cite{street1987algebra}. 
\end{exm}

\begin{dfn} [Stricter \( \omega \)\nbd category]
    Let \( C \) be a strict \( \omega \)\nbd category.
    We say that \( C \) is a \emph{stricter \( \omega \)\nbd category} if for all molecules \( U \) and \( V \), all \( k \in \nat \) such that \( U \cp{k} V \) is defined, \( C \) is local with respect to the canonical strict functor
    \begin{equation*}
        \molecin{U} \cup_{\molecin{\bd{k}{+} U}} \molecin{V} \to \molecin{U \cp{k} V}.
    \end{equation*}
    We write \( \somegaCat \) for the full subcategory of \( \omegaCat \) on stricter \( \omega \)\nbd categories.
\end{dfn}

\begin{lem} \label{lem:directed_complex_are_stricter}
    Let \( X \) be a directed complex.
    Then \( \molecin{X} \) is a stricter \( \omega \)\nbd category.
\end{lem}
\begin{proof}
    The same proof as \cite[Lemma 3.6]{chanavat2025stricter} applies.
\end{proof}

\begin{lem} \label{lem:strict_functor_preserve_basis_are_cellular}
    Let \( X, Y \) be directed complexes and \( \F \colon \molecin{X} \to \molecin{Y} \) be a strict functor such that for all cells \( x \colon U \to X \), the pasting diagram \( \F(x) \) is a cell of \( Y \) with \( \dim \F(x) = \dim x \).
    Then there exists a morphism \( f \colon X \to Y \) of directed complexes such that \( \F = \molecin{f} \).
\end{lem}
\begin{proof}
    We prove, by induction on \( n \geq -1 \), that the statement holds of the \( n \)\nbd skeleton \( \skel{n}{\F} \colon \molecin{\skel{n}{X}} \to \molecin{\skel{n}{Y}} \) of \( \F \).
    When \( n = -1 \), the statement is clear.
    Inductively, let \( n > 0 \), and suppose that \( \skel{n - 1}{\F} = \molecin{\skel{n - 1}{f}} \), for a morphism \( \skel{n - 1}{f} \colon \skel{n - 1}{X} \to \skel{n - 1}{Y} \). 
    Let \( x \colon U \to X \) be an \( n \)\nbd cell of \( X \), write \( \bd{}{} x \) for the restriction of \( x \) along the total boundary \( \bd{n - 1}{} U \incl U \) of \( U \).
    Then by inductive hypothesis, \( \bd{}{}\F(x) = \molecin{\skel{n - 1}{f}}(\bd{}{}x) \colon \bd{}{} U \to Y \).
    Since \( \F \) is dimension-preserving, the cell \( \F(x) \) has type \( U \to Y \), and we may define \( \skel{n}{f}(x) \eqdef \F(x) \), which extends \( \skel{n - 1}{f} \) to a well-defined morphism \( \skel{n}{f} \colon \gr{n}{X} \to \gr{n}{Y} \) of directed complexes such that \( \skel{n}{\F} = \molecin{\skel{n}{f}} \).
    This concludes the induction and the proof. 
\end{proof}

\begin{lem} \label{lem:cellular_colimits_are_omega_colimits}
    The functor \( \molecin{-} \colon \dcpx \to \somegaCat \) is a pseudomonic left-adjoint.
\end{lem}
\begin{proof}
    First, by Lemma \ref{lem:directed_complex_are_stricter}, \( \molecin{-} \colon \dcpx \to \omegaCat \) factors through \( \somegaCat \).
    We now show that \( \molecin{-} \) is left adjoint, or equivalently, since \( \somegaCat \) is cocomplete, that \( \molecin{-} \) is the left Kan extension of its restriction along the Yoneda embedding \( \atomcat \incl \dcpx \).
    Since any molecule is finite, \( \molecin{-} \) preserves transfinite compositions.
    Hence by Definition \ref{dfn:skeleton_dcpx}, it is enough to show that \( \molecin{-} \) preserve pushouts of the form
    \begin{center}
        \begin{tikzcd}
            {\bd{}{} U} & U \\
            X & Y
            \arrow[""{name=0, anchor=center, inner sep=0}, hook, from=1-1, to=1-2]
            \arrow[from=1-1, to=2-1]
            \arrow[from=1-2, to=2-2]
            \arrow[hook, from=2-1, to=2-2]
            \arrow["\lrcorner"{anchor=center, pos=0.125, rotate=180}, draw=none, from=2-2, to=0]
        \end{tikzcd}
    \end{center}
    where \( U \) is an atom and \( X \) and \( Y \) are directed complexes.
    Since the category of elements functor preserves colimits, the previous pushout is preserved by \( \cell \).
    Since the objects \( \cell{Y} \) are a basis of \( \molecin{Y} \), this shows uniqueness of the universal property.
    For existence, given a stricter \( \omega \)\nbd category \( C \) any pair of strict functors \( \F \colon \molecin{X} \to C \) and \( \G \colon \molecin{U} \to C \) define, by the previous point, for each cell \( y \colon U \to Y \), a strict functor \( \fun{H}_y \colon \molecin{U} \to C \).
    For each pasting diagram \( f \colon V \to Y \), the collection \( \set{\fun{H}_{f \after \iota_x}}_{x \in V} \) is a matching family in the sense of \cite[2.14]{chanavat2025stricter}, hence we may conclude by a simple variant of \cite[Lemma 2.18]{chanavat2025stricter}.
    
    Now let \( f, f' \colon X \to Y \) be two morphisms of directed complexes such that \( \molecin{f} = \molecin{f'} \), and let \( x \colon U \to X \) be a cell in \( X \).
    Then 
    \begin{equation*}
       f(x) = \molecin{f}(x) = \molecin{f'}(x) = f'(x). 
    \end{equation*}
    Thus \( f = f' \).
    Next, reasoning as in the proof of \cite[Proposition 6.2.37]{hadzihasanovic2024combinatorics}, any isomorphism \( \F \colon \molecin{X} \to \molecin{Y} \) has to send the object of \( \cell X \) to the objects of \( \cell Y \) with the same dimension.
    By Lemma \ref{lem:strict_functor_preserve_basis_are_cellular}, \( \F \) is of the form \( \molecin{f} \) for a morphism \( f \colon X \to Y \) for directed complexes, which is also an isomorphism, since the inverse of \( \F \) is also of the form \( \molecin{f'} \) by the same reasoning.
    This shows at once that \( \molecin{-} \) reflects, and is full on, isomorphisms.
    This concludes the proof.  
\end{proof}

\begin{comm}
    In general, the subcategory inclusion \( \dcpx \incl \omegaCat \) does \emph{not} preserve colimits.
\end{comm}

\noindent The previous Lemma indicates that we can identify \( \dcpx \) with the wide subcategory of \( \dcpxomega \) on strict functors of the form \( \molecin{f} \), for \( f \) a morphism of directed complexes.
Together with \cite[Theorem 6.2.35]{hadzihasanovic2024combinatorics}, we have the following diagram of subcategories
\begin{center}
    \begin{tikzcd}
        \rdcpxequal & \rdcpxmap \\
        \dcpx & \dcpxomega.
        \arrow[hook, from=1-1, to=1-2]
        \arrow[hook, from=1-1, to=2-1]
        \arrow[hook, from=1-2, to=2-2]
        \arrow[hook, from=2-1, to=2-2]
    \end{tikzcd}
\end{center}
In particular, we may import terminology that applies to morphisms of directed complexes, or maps of regular directed complexes, to strict functors of the form \( \molecin{f} \).
For instance, we may say that a functor \( \F \colon \molecin{U} \to \molecin{X} \) is a pasting diagram if it is of the form \( \molecin{f} \), for a pasting diagram \( f \colon U \to X \).
If instead \( U \) and \( X \) are regular directed complexes, and \( f \colon U \to X \) is a surjective map, then we would say that \( \F \) is a collapse.

Next, given a strict functor \( \F \colon \molecin{X} \to \molecin{Y} \) and a diagram \( f \colon U \to P \), then \( \F(f) \) only a priori makes sense when \( U \) is a molecule, since then, \( f \) is a pasting diagram.
The next definition extends this definition when \( U \) is an arbitrary directed complex.

\begin{dfn} [Evaluation at a diagram]
    Let \( \F \colon \molecin{X} \to \molecin{Y} \) be a strict functor of directed complexes. 
    Then for any pasting diagram \( f^- \cp{k} f^+ \colon U^- \cp{k} U^+ \to X \) of \( X \) and for each \( \a \in \set{-, +} \), there is a dashed submolecule inclusion making the triangle
    \begin{center}
        \begin{tikzcd}[column sep=large]
            {\F(U^\a)} & \\
            {\F(U^-) \cp{k} \F(U^+)} & Y
            \arrow["{F(\iota^\a)}"', dashed, from=1-1, to=2-1]
            \arrow["{\F(f^\a)}", from=1-1, to=2-2]
            \arrow["{\F(f^- \cp{k} f^+)}"', from=2-1, to=2-2]
        \end{tikzcd}
    \end{center}
    commute.
    Since any embedding of atoms is a composition of submolecule inclusions of the form \( \iota^\a \), this assignment extends by induction on the dimension to a functor of type \( \cell X \to \Pd Y \).
    Since \( \slice{\dcpx}{X} \cong \Psh(\cell X) \) and \( \Pd Y \) is a subcategory of \( \slice{\dcpx}{Y} \), this functor further extends, by left Kan extension, to a cocontinuous functor of type
    \begin{equation*}
        \slice{\dcpx}{X} \to \slice{\dcpx}{Y},
    \end{equation*}
    which agrees with \( \F \) on pasting diagrams.
    Thus given a diagram \( f \colon U \to X \) in \( X \), we will again write \( \F(f) \colon \F(U) \to Y \) for the diagram in \( Y \) given by the action on object of that functor, and given a diagram \( g \colon V \to U \), we write \( \restr{\F}{f}(g) \) for its action on the morphism \( g \colon f \after g \to f \) of \( \slice{\dcpx}{X} \).

    This construction defines a strict functor \( \restr{\F}{f} \colon \molecin{U} \to \molecin{\F(U)} \), called the \emph{restriction of \( \F \) along \( f \)}, and fits in the following commutative square of strict functors:
    \begin{center}
        \begin{tikzcd}
            {\molecin{U}} & {\molecin{\F(U)}} \\
            \molecin{X} & \molecin{Y}.
            \arrow["{\restr{\F}{f}}", from=1-1, to=1-2]
            \arrow["{\molecin{f}}"', from=1-1, to=2-1]
            \arrow["{\molecin{\F(f)}}", from=1-2, to=2-2]
            \arrow["\F"', from=2-1, to=2-2]
        \end{tikzcd}
    \end{center}
    We may also write \( \restr{\F}{U} \) for \( \restr{\F}{f} \) when \( f \) is clear from the context.
\end{dfn}

\begin{dfn} [Principal diagram]
    Let \( \F \colon \molecin{U} \to \molecin{V} \) be a strict functor of directed complexes.
    The \emph{principal diagram} of \( \F \) is the diagram
    \begin{equation*}
        \pcell{\F} \eqdef \F(\idd{U}) \colon \F(U) \to V.
    \end{equation*}
    When \( U \) is a molecule, then the principal diagram is an arrow of \( \molecin{V} \), which we call the \emph{principal pasting diagram} of \( \F \).
\end{dfn}

\begin{dfn} [Quotient-free and active strict functor] \label{dfn:active_quotient-free}
    Let \( \F \colon \molecin{X} \to \molecin{Y} \) be a strict functors of directed complexes.
    We say that \( \F \) is \emph{quotient-free} if the principal diagram \( \pcell{\F} \) is an embedding.
    If furthermore \( \pcell{\F} = \idd{Y} \), we say that \( \F \) is \emph{active}.
    The \emph{active part of \( \F \)} is the active strict functor \( \hat{\F} \eqdef \restr{\F}{\pcell{\F}} \), which fits in the following diagram
    \begin{center}
        \begin{tikzcd}
            & \molecin{\F(X)} \\
            \molecin{X} && \molecin{Y}.
            \arrow["{\molecin{\pcell \F}}", from=1-2, to=2-3]
            \arrow["{\hat\F}", from=2-1, to=1-2]
            \arrow["\F"', from=2-1, to=2-3]
        \end{tikzcd}
    \end{center}
\end{dfn}

\begin{rmk}
    The restriction \( \restr{\F}{f} \) of \( \F \) along any diagram \( f \colon U \to X \) is active.
\end{rmk}

\begin{exm}
    A morphism \( [k] \to [n] \) in the simplex category is active in the usual sense if and only if its associated strict functor \( \molecin{k\arr} \to \molecin{n\arr} \) is active in the sense of the previous definition.
\end{exm}

\begin{lem} \label{lem:active_strict_functor_equal_iff_same_shape}
    Let \( \F, \G \colon \molecin{U} \to \molecin{V} \) be two active strict functors of molecules such that for all submolecules \( \iota \colon W \submol U \), the molecules \( \F(W) \) and \( \G(W) \) are equal.
    Then \( \F = \G \).
\end{lem}
\begin{proof}
    We show by induction on the submolecules \( \iota \colon W \submol U \) of \( U \) that \( \restr{\F}{\iota}, \restr{\G}{\iota} \colon W \to \F(W) = \G(W) \) are equal.
    This is clear for the base case where \( \iota \) is of type \( \pt \submol U \), since both \( \restr{\F}{\iota} \) and \( \restr{\G}{\iota} \) are the identity.
    Suppose the statement holds of all proper submolecules of \( U \).
    If \( U \) is an atom of dimension \( n \), we have the equality
    \begin{equation*}
        \restr{\F}{\bd{}{\a} U} = \restr{\G}{\bd{}{\a} U} \colon \molecin{\bd{}{\a} U} \to \molecin{\bd{n - 1}{\a} V},
    \end{equation*}
    since these two strict functors agree on the objects of \( \cell \bd{}{\a} U \) by the inductive hypothesis.
    Since \( \F \) and \( \G \) are active, \( \F(\idd{U}) = \G(\idd{U}) \).
    Thus, \( \F = \G \).

    Now suppose that \( U \) splits as \( A \cp{k} B \), for proper submolecules \( \iota_A \colon A \submol U \) and \( \iota_B \colon B \submol U \). 
    By inductive hypothesis, \( \restr{\F}{\iota_A} = \restr{\G}{\iota_A} \) and \( \restr{\F}{\iota_B} = \restr{\G}{\iota_B} \).
    To conclude, it is therefore enough to show that \( \F(\iota_A) = \G(\iota_A) \) and \( \F(\iota_B) = \G(\iota_B) \).
    Since \( \F \) and \( \G \) are active strict functors, we have 
    \begin{equation*}
        \F(A) \cp{k} \F(B) = V = \G(A) \cp{k} \G(B).
    \end{equation*}
    Since furthermore \( \F(A) = \G(A) \) and \( \F(B) = \G(B) \), the two embeddings 
    \begin{equation*}
         \F(\iota_A) \colon \F(A) \incl \F(A) \cp{k} \F(B) \quad\text{ and }\quad \G(\iota_A) \colon \G(A) \incl \G(A) \cp{k} \G(B)
    \end{equation*}
    are equal. 
    Similarly \( \F(\iota_B) = \G(\iota_B) \).
    This concludes the proof. 
\end{proof}

\begin{prop} \label{prop:ofs_functor_of_dcpx}
    The category \( \dcpxomega \) has an orthogonal factorisation system whose left class consists of the active strict functors and whose right class consists of the diagrams.
\end{prop}
\begin{proof}
    Consider a strict functor \( \F \colon \molecin{X} \to \molecin{Y} \) of directed complexes.
    By Definition \ref{dfn:active_quotient-free}, we have \( \F = \molecin{\pcell{\F}} \after \hat \F \), where \( \hat \F \) is the active part of \( \F \), and \( \pcell{\F} \) is indeed a diagram.
    This gives the existence of the factorisation.
    For uniqueness (up to unique isomorphism), consider another factorisation
    \begin{center}
        \begin{tikzcd}
            & \molecin{Z} \\
            \molecin{X} && \molecin{Y}
            \arrow["{\molecin{f}}", from=1-2, to=2-3]
            \arrow["\G", from=2-1, to=1-2]
            \arrow["\F"', from=2-1, to=2-3]
        \end{tikzcd}
    \end{center}
    where \( f \colon Z \to Y \) is a diagram and \( \G \) is active.
    Then, since \( \G \) is active, up to the unique isomorphism given when computing the extension of \( \F \) to \( \slice{\rdcpxequal}{X} \), we have
    \begin{equation*}
        \pcell{\F} = f \after \G(\idd{X}) = f.
    \end{equation*}
    It remains to show that \( \G = \hat{\F} \).
    Let \( x \colon U \to X \) be a cell of \( X \). 
    For all submolecules \( j \colon W \submol U \) of \( U \), we have by assumption 
    \begin{equation*}
        \pcell{\F} \after \G(x \after j) = \F(x \after j) = \pcell{\F} \after \hat{\F}(x \after j) \colon \F(W) \to Y.
    \end{equation*}
    Thus the functors \( \restr{G}{x}, \restr{\hat{\F}}{x} \colon \molecin{U} \to \molecin{\F(U)} \) satisfy the hypothesis of Lemma \ref{lem:active_strict_functor_equal_iff_same_shape}, hence are equal. 
    In particular, \( \G(x) = \pcell{\restr{\G}{x}} = \pcell{\restr{\hat{\F}}{x}} = \hat{\F}(x) \).
    Since \( x \) was arbitrary, Remark \ref{rmk:atom_basis_mol} implies that \( \G = \hat{\F} \).
    This shows uniqueness of the factorisation and concludes the proof.
\end{proof}

\begin{cor} \label{cor:ofs_functor_of_rdcpx}
    The full subcategory of \( \dcpxomega \) on regular directed complexes has an orthogonal factorisation system whose left class consists of the active strict functors and whose right class consists of the local embeddings.
    This orthogonal factorisation system restricts to the further full subcategory of \( \dcpxomega \) on molecules.
\end{cor}
\begin{proof}
    If \( \F \colon \molecin{U} \to \molecin{V} \) is a strict functor of regular directed complexes, then \( \pcell{\F} \colon \F(U) \to V \) is a local embedding of regular directed complexes, where \( \F(U) \) is a molecule if \( U \) is a molecule.
    We conclude by Proposition \ref{prop:ofs_functor_of_dcpx}.
\end{proof}

\begin{rmk}
    The full subcategory of \( \dcpxomega \) on \( 1 \)\nbd molecules is canonically isomorphic to the simplex category.
    The previous Proposition then restricts to the usual active-inert factorisation system.
    More generally (see Proposition \ref{prop:ternary_factorisation_of_theta}), the previous Proposition recovers the \emph{algebraic-globular} orthogonal factorisation system \cite[Proposition 3.3.10]{ara2010groupoid} on the theta category \( \Theta \). 
\end{rmk}

\begin{lem} \label{lem:quotient-free_functor_send_embedding_to_embedding}
    Let \( \F \colon \molecin{X} \to \molecin{Y} \) be a quotient-free strict functor of directed complexes and \( \iota \colon A \incl X \) be an embedding.
    Then \( \F(\iota) \) is an embedding.
\end{lem}
\begin{proof}
    Since \( \F \) is quotient-free, \( \F(\iota) \) is the composite of the canonical embedding
    \begin{equation*}
       \colim_{x \colon U \to A} \F(U) \to \colim_{x \colon U \to X} \F(U)
    \end{equation*}
    together with the embedding \( \pcell{\F} \), hence is an embedding.
\end{proof}

\begin{lem} \label{lem:active_functor_preserve_submolecule}
    Let \( \F \colon \molecin{U} \to \molecin{V} \) be an active strict functor of molecules, and \( \iota \colon W \submol U \) be a submolecule of \( U \).
    Then \( \F(\iota) \) is a submolecule of \( V \).
\end{lem}
\begin{proof}
    First, Lemma \ref{lem:quotient-free_functor_send_embedding_to_embedding} shows that \( \F(\iota) \) is an embedding.
    We show by induction on the submolecule \( \iota \colon W \submol U \) of \( U \) that the active strict functor \( \restr{\F}{\iota} \colon \molecin{W} \to \molecin{\F(W)} \) sends submolecules to submolecules.
    The base case where \( \iota \) is of type \( \pt \submol U \) is clear since \( \restr{\F}{\iota} \) is the identity.
    Inductively, suppose the statement holds of all proper submolecules of \( U \).
    Let \( \iota \colon W \submol U \) be a submolecule of \( U \).
    If \( \iota \) is the identity, then we are done since \( \F(\iota) = \idd{V} \) is a submolecule of \( V \).
    Otherwise, we may write \( \iota = \iota^\a \after \iota' \), where \( \iota^\a \colon U^\a \to U^- \cp{k} U^+ = U \) is a proper submolecule of \( U \) for some \( \a \in \set{-, +} \), and \( \iota' \) is another submolecule.
    Then, the triangle
    \begin{center}
        \begin{tikzcd}[row sep=2.25em]
            {\F(W)} \\
            {\F(U^\a)} & {\F(U^-) \cp{k} \F(U^+) = V}
            \arrow["{\restr{\F}{\iota^\a}(\iota')}"', from=1-1, to=2-1]
            \arrow["{\F(\iota)}", from=1-1, to=2-2]
            \arrow[hook, from=2-1, to=2-2]
        \end{tikzcd}
    \end{center}
    commutes.
    Since the bottom embedding is a submolecule of \( V \), as is \( \restr{\F}{\iota^\a}(\iota') \) by inductive hypothesis, we conclude that \( \F(\iota) \) is a submolecule of \( V \).
\end{proof}

\subsection{Comparison with regular polygraphs}

\noindent We conclude by comparing (positive) regular polygraphs \cite{henry2018regular} with directed complexes. 
Our strategy is to extract from each polyplex an underlying poset (which is just its category of elements), and show that it is graded and canonically oriented (Lemma \ref{lem:underlying_poset_is_graded}).
We then show that this oriented graded poset is always a molecule in Proposition \ref{prop:posplex_sends_to_molecules}, and from there, we quickly reach Theorem \ref{thm:from_polyplexes_to_molecules}, which shows in particular that regular polygraphs of dimension \( \le 3 \) coincide with strict \( \omega \)\nbd categories of the form \( \molecin{X} \), for \( X \) a directed \( 3 \)\nbd complex.

\begin{dfn} [Polygraph] \label{dfn:polygraph}
    Let \( n \in \nat \cup \set{\omega} \) and \( C \) be a strict \( n \)\nbd category.
    We say that \( C \) is an \emph{\( n \)\nbd polygraph} if, for each \( d \in \nat \), there is a collection \( \cls{S}_d \) of \( d \)\nbd arrows of \( C \), called the \emph{generating \( d \)\nbd arrows}, such that the diagram
    \begin{center}
        \begin{tikzcd}
            {\coprod\limits_{x \in \cls{S}_d}\bd{}{}\globe{d}} & {\coprod\limits_{x \in \cls{S}_d}\globe{d}} \\
            {\skel{d - 1}{C}} & {\skel{d}{C}}
            \arrow[from=1-1, to=1-2]
            \arrow[from=1-1, to=2-1]
            \arrow[from=1-2, to=2-2]
            \arrow[from=2-1, to=2-2]
        \end{tikzcd}
    \end{center}
    is a pushout square in \( \nCat{n} \), exhibiting \( \skel{d} C \) as \emph{cellular extension} of \( \skel{d - 1} C \).
    If \( n = \omega \) or is irrelevant, we speak simply of polygraph.
\end{dfn}

\begin{comm}
    To be fully precise, we should not say that \( C \) is a polygraph, but rather that \( C \) \emph{is free on a polygraph}.
    We decided to make no distinction in terminology, since being free on a polygraph is a mere property \cite[Proposition 16.6.3]{ara2025polygraphs}.
\end{comm}

\noindent We now give the basic definition and results of \cite{henry2018nonunital,henry2018regular}.

\begin{dfn} [Positive polygraphs, plexes, and polyplexes]
    A \emph{positive polygraph} is a polygraph such that the source and target of each generating \( d \)\nbd arrow is an arrow of dimension \( d - 1 \).
    A \emph{morphism of positive polygraphs} is a strict functor sending generating arrows to generating arrows of the same dimension.
    There is a family of positive polygraphs called the \emph{polyplexes} with a subfamily called the \emph{plexes} such that for each positive polygraph \( X \), there is a natural bijection between the arrows of \( X \) and the set of morphisms of polygraphs \( \pl{u} \to X \), for \( \pl{u} \) ranging over polyplexes.
    This restricts to a bijection between generating arrows of \( X \) and morphisms whose domain is a plex. 
    In particular, each polyplex \( \pl{u} \) has a special arrow, called its \emph{universal arrow}, and written \( u \), which is given by the image of \( \idd{\pl{u}} \) along the bijection.
\end{dfn}

\begin{dfn} [Boundaries and pasting of polyplexes]
    It can be shown that there exists a terminal positive polygraph, hence that polyplexes are in bijection with arrows of the terminal object.
    In particular, they inherit a structure of strict \( \omega \)\nbd category, whose \( k \)\nbd composition operation we write \( - \cp{k} - \).
    Let \( \pl{u} \) be a polyplex.
    Then for each \( k \in \nat \) and \( \a \in \set{-, +} \), we may write \( \bd{k}{\a} \pl{u} \) for the polyplex representing the boundary of \( \pl{u} \) as an arrow of the terminal object. 
    There is a canonical morphism of polyplexes \( \bd{k}{\a} \pl{u} \to \pl{u} \) sending the universal arrow of \( \bd{k}{\a} \pl{u} \) to \( \bd{k}{\a} u \).
    Thus, given a morphism \( x \colon \pl{u} \to X \) where \( X \) is a positive polygraph, we may write \( \bd{k}{\a} x \colon \bd{k}{\a} \pl{u} \to X \) for the precomposition of \( x \) with \( \bd{k}{\a} \pl{u} \to \pl{u} \).
    This operation is compatible with the boundary of \( x \) seen as an arrow of \( X \).
    
    If \( \pl{v} \) is another polyplex such that there exists a, necessarily unique, isomorphism \( \bd{k}{+} \pl{u} \cong \bd{k}{-} \pl{v} \), then \( \pl{u} \cp{k} \pl{v} \) is defined, and there is a commutative square
    \begin{center}
        \begin{tikzcd}
            {\bd{k}{+} \pl{u} \cong \bd{k}{+} \pl{v}} & {\pl{v}} \\
            {\pl{u}} & {\pl{u} \cp{k} \pl{v}}
            \arrow[from=1-1, to=1-2]
            \arrow[from=1-1, to=2-1]
            \arrow[from=1-2, to=2-2]
            \arrow[from=2-1, to=2-2]
        \end{tikzcd}
    \end{center}
    which is a pushout in \( \omegaCat \).
\end{dfn}

\begin{dfn} [Polyplexes with spherical boundaries and regular polygraphs]
    Let \( \pl{u} \) be a polyplex.
    We say that \( \pl{u} \) has \emph{spherical boundaries} if, for all \( k \in \nat \) and all \( \a \in \set{-, +} \), the comparison morphism
    \begin{equation*}
        \bd{k}{-} \pl{u} \coprod_{\bd{k - 1}{-}\pl{u} \coprod \bd{k}{+} \pl{u}} \bd{k}{+} \pl{u} \longrightarrow \bd{k + 1}{\a} \pl{u}
    \end{equation*} 
    is a monomorphism, with convention that \( \bd{-1}{\a} \pl{u} = \varnothing \).
    A positive polygraph is called a \emph{regular polygraph} if all of its plexes have spherical boundaries.

    We write \( \Pol \) for the full subcategory of positive polygraphs and their morphisms on regular polygraphs.
    Given \( n \in \nat \cup \set{\omega} \), we write \( \nPol{n} \) for the subcategory of \( \Pol \) on the \( n \)\nbd polygraphs, and \( \nplexcat{n} \) and \( \npolyplexcat{n} \) for its further full subcategories on plexes and polyplexes, or simply \( \plexcat \) and \( \polyplexcat \) when \( n = \omega \).

    The main result of \cite{henry2018nonunital} implies that \( \nPol n \) is equivalent to the category of presheaves over \( \nplexcat n \), and the inclusion \( \nplexcat n \incl \nPol n \) is the Yoneda embedding.
\end{dfn}

\begin{rmk}
    If \( \pl{u} \) is a regular polyplex, then for all \( k \in \nat \) and \( \a \in \set{-, +} \), the comparison morphism \( \bd{k}{\a} \pl{u} \to \pl{u} \) is a monomorphism, hence so are morphisms of the form \( \pl{u} \to \pl{u} \cp{k} \pl{v} \) and \( \pl{v} \to \pl{u} \cp{k} \pl{v} \).
\end{rmk}

\begin{dfn} [Round polyplex]
    A \emph{round polyplex} is a regular polyplex with spherical boundaries.
\end{dfn}

\begin{exm}
    All regular plexes are round.
\end{exm}

\begin{lem} \label{round_polyplex}
    Let \( \pl{u} \) be a regular polyplex of dimension \( d \in \nat \).
    Then \( \pl{u} \) is round if for all \( k < d \), the canonical comparison
    \begin{equation*}
        \bd{k - 1}{-} \pl{u} \cup \bd{k - 1}{+} \pl{u} \to \bd{k}{+} \pl{u} \cap \bd{k}{-} \pl{u}
    \end{equation*} 
    of subobjects of \( \pl{u} \) is an isomorphism.
\end{lem}
\begin{proof}
    This is \cite[Corollary 2.4.4]{henry2018regular}.
\end{proof}

\begin{dfn} [Underlying poset of a polyplex]
    Let \( \pl{p} \) be a regular polyplex.
    The \emph{underlying poset of \( \pl{p} \)} is the category of elements of \( \pl{p} \). 
\end{dfn}
te
\begin{rmk}
    By \cite[Proposition 2.4.5]{henry2018regular}, the category of elements of \( \pl{p} \) is indeed a poset, whose objects are the plexes \( x \colon \pl{v} \to \pl{p} \) of \( \pl{p} \), and such that \( x \le y \) if \( y \) factors through \( x \) (which can happen at most once).
\end{rmk}

\begin{dfn} [Dimension in the underlying poset]
    Let \( \pl{p} \) be a regular polyplex, and \( x \colon \pl{u} \to \pl{p} \) be an element of its underlying poset.
    The \emph{dimension of \( x \)} is \( \dim x \eqdef \dim \pl{u} \).
\end{dfn}

\begin{lem} \label{lem:maximal_of_up_is_in_boundary}
    Let \( \pl{p} \) be a regular polyplex, and \( x \) be a maximal element of the underlying poset of \( \pl{p} \), and write \( k \eqdef \dim x \).
    Then \( x \in \bd{k}{-} \pl{p} \cap \bd{k}{+} \pl{p} \).
\end{lem}
\begin{proof}
    We proceed by induction on the construction of the regular polyplex.
    If \( \pl{p} \) is a plex, then \( \idd{\pl{p}} \) is the greatest element of the underlying poset of \( \pl{p} \), thus \( x = \idd{\pl{p}} \) and the statement follows.
    Otherwise, \( \pl{p} = \pl{u} \cp{r} \pl{v} \).
    We suppose that \( x \in \pl{u} \), the other case is entirely dual.
    If \( r < k  \), then \( \bd{k}{\a} \pl{p} = \bd{k}{\a} \pl{u} \cp{r} \bd{k}{\a} \pl{v} \) for all \( \a \in \set{-, +} \).
    Since \( x \) is maximal in \( \pl{u} \), we may conclude by inductive hypothesis.
    
    Now suppose \( r \geq k \).
    Then \( x \) is maximal in \( \pl{u} \), thus \( x \in \bd{k}{-} \pl{u} \cap \bd{k}{+} \pl{u} \) by inductive hypothesis.
    By globularity, \( \bd{k}{+} \pl{u} \subseteq \bd{r}{+} \pl{u} \), so \( x \in \pl{v} \).
    Since \( x \) is maximal in \( \pl{v} \), by inductive hypothesis, \( x \in \bd{k}{+} \pl{v} \).
    Thus 
    \begin{equation*}
        x \in \bd{k}{-} \pl{u} \cap \bd{k}{+} \pl{v} = \bd{k}{-} \pl{p} \cap \bd{k}{+} \pl{p}.
    \end{equation*}
    This concludes the induction and the proof.
\end{proof}

\begin{lem} \label{lem:round_polyplex_are_pure}
    Let \( \pl{p} \) be a round polyplex and \( x \colon \pl{u} \to \pl{p} \) be a maximal element of the underlying poset of \( \pl{p} \).
    Then \( \dim x = \dim \pl{p} \).
\end{lem}
\begin{proof}
    Since \( x \) is maximal, letting \( k \eqdef \dim x \), we have by Lemma \ref{lem:maximal_of_up_is_in_boundary} that 
    \begin{equation*}
        x \in \bd{k}{-} \pl{p} \cap \bd{k}{+} \pl{p}.
    \end{equation*}
    If \( k < \dim \pl{p} \), the spherical boundaries of \( \pl{p} \) imply that \( x \in \bd{k - 1}{-} \pl{p} \cup \bd{k - 1}{+} \pl{p} \), which contains no plex of dimension \( \geq k \), a contradiction.
    Therefore, \( k = \dim \pl{p} \).
\end{proof}

\begin{lem} \label{lem:underlying_poset_is_graded}
    Let \( \pl{p} \) be a regular polyplex.
    Then the underlying poset of \( \pl{p} \) is graded by \( x \mapsto \dim x \).
    Furthermore, if \( y \) covers \( x \), then there exists a unique \( \a \in \set{-, +} \) such that \( x \) factors through \( \bd{\dim x}{\a} y \incl y \). 
\end{lem}
\begin{proof}
    We show the statement by induction on the dimension and construction of the polyplex.
    If \( \pl{p} = \pl{q} \cp{k} \pl{r} \), then, since the category of elements preserves colimits, the underlying poset of \( \pl{p} \) is a pushout of closed embeddings of graded posets, thus is graded, and the second condition follows immediately by inductive hypothesis.
    If \( \pl{p} \) is a plex, then the universal arrow \( p = \idd{\pl{p}} \colon \pl{p} \to \pl{p} \) is the maximal element of the underlying poset. 
    Since \( \pl{p} \) is regular, the underlying poset of \( \pl{p} \) without its maximal element is constructed as a pushout of closed embeddings of the underlying posets of \( \bd{k}{\a} \pl{p} \), for \( \a \in \set{-, +} \) and \( k < \dim p \).
    It follows that the statement holds of that poset.
    Next, we show that \( p \) only covers elements of dimension \( k \eqdef \dim p - 1 \).
    Let \( x \colon \pl{u} \to \pl{p} \) be covered by \( p \).
    Then there exist \( \a \in \set{-, +} \) such that \( x \in \bd{k}{\a} \pl{p} \), for if not we would have \( x = p \).
    Since \( p \) covers \( x \), \( x \) is maximal in \( \bd{k}{\a} \pl{p} \), and because \( \bd{k}{\a} \pl{p} \) is round, Lemma \ref{lem:round_polyplex_are_pure} implies that \( \dim x = k \).
    This shows that the underlying poset of \( \pl{p} \) is graded.
    Last, if \( x \in \bd{k}{-} \pl{p} \cap \bd{k}{+} \pl{p} \), we conclude by spherical boundaries that \( \dim x < k \), which shows uniqueness.
\end{proof}

\begin{dfn} [Oriented graded poset of a polyplex]
    Let \( \pl{p} \) be a regular polyplex.
    By Lemma \ref{lem:underlying_poset_is_graded}, the underlying poset of \( \pl{p} \) is graded by \( x \mapsto \dim x \), and has a canonical orientation given by \( x \in \faces{}{\a} y \) if \( x \) factors through \( \bd{\dim x}{\a} y \incl y \).
    We write \( \posplex(\pl{p}) \) for this oriented graded poset, called the \emph{oriented graded poset of \( \pl{p} \)}.
    This mapping extends to a functor
    \begin{equation*}
        \posplex \colon \polyplexcat \to \cat{ogPos}
    \end{equation*}
    from the category of polyplexes to the category \( \cat{ogPos} \) of oriented graded posets and their morphisms.
\end{dfn}

\begin{prop} \label{prop:posplex_sends_to_molecules}
    Let \( \pl{p} \) be a polyplex.
    Then \( \posplex(\pl{p}) \) is a molecule such that:
    \begin{enumerate}
        \item for all \( \a \in \set{-, +} \) and \( k \in \nat \), \( \posplex(\bd{k}{\a} \pl{p}) = \bd{k}{\a} \posplex(\pl{p}) \), as closed subsets of \( \posplex(\pl{p}) \);
        \item if \( \pl{p} = \pl{u} \cp{k} \pl{v} \), for some \( k \in \nat \), and polyplexes \( \pl{u} \) and \( \pl{v} \), then \( \posplex(\pl{u} \cp{k} \pl{v}) \) and \( \posplex(\pl{u}) \cp{k} \posplex(\pl{v}) \) are uniquely isomorphic;
        \item if \( \pl{p} \) is round, then \( \posplex(\pl{p}) \) is round.
    \end{enumerate}    
\end{prop}
\begin{proof}
    We proceed by induction on the dimension and the construction of the polyplex \( \pl{p} \).
    When \( \pl{p} \) is \( 0 \)\nbd dimensional, the statement is clear.
    Now suppose \( \pl{p} \) is a plex of dimension \( d > 0 \).
    Then, \( \posplex(\bd{d - 1}{-} \pl{p}) \) and \( \posplex(\bd{d - 1}{+} \pl{p}) \) are, by inductive hypothesis, parallel round molecules of dimension \( d - 1 \).
    Thus \( \posplex(\bd{d - 1}{-} \pl{p}) \celto \posplex(\bd{d - 1}{+} \pl{p}) \) is a well-defined atom, which is, by construction, isomorphic to \( \posplex(\pl{p}) \).
    By construction, and using the inductive hypothesis, for all \( \a \in \set{-, +} \) and \( k \in \nat \), \( \posplex(\bd{k}{\a} \pl{p}) = \bd{k}{\a} \posplex(\pl{p}) \).
    Now suppose that \( \pl{p} = \pl{u} \cp{k} \pl{v} \), then since \( \posplex(-) \) commutes with the boundaries of \( \pl{u} \) and \( \pl{v} \) and the category of elements functor preserves pushouts, we immediately get that \( \posplex(\pl{u} \cp{k} \pl{v}) \) is isomorphic to \( \posplex(\pl{u}) \cp{k} \posplex(\pl{v}) \).
    Since the latter is a molecule, as pasting of molecules by inductive hypothesis, we conclude both that \( \posplex(\pl{u} \cp{k} \pl{v}) \) is a molecule and that the isomorphism is unique.
    The statement on boundaries now follows from inductive hypothesis and the axioms of distributivity of boundary over pasting in a strict \( \omega \)\nbd category.

    Finally, suppose that \( \pl{p} \) is round, and let \( k < \dim \pl{p} \).
    We must show that 
    \begin{equation*}
       \bd{k - 1}{-} \posplex(\pl{p}) \cup \bd{k - 1}{+} \posplex(\pl{p}) = \bd{k}{-} \posplex(\pl{p}) \cap \bd{k}{+} \posplex(\pl{p}).  
    \end{equation*}
    Since all the terms involved in the previous equation are underlying graded posets of molecules, they are globular, so we already have one inclusion.
    Conversely, let \( x \colon \pl{u} \to \pl{p} \) be in 
    \begin{equation*}
       \bd{k}{-} \posplex(\pl{p}) \cap \bd{k}{+} \posplex(\pl{p}) = \posplex(\bd{k}{-} \pl{p}) \cap \posplex(\bd{k}{+} \pl{p}). 
    \end{equation*}
    Then \( x \in \bd{k}{-} \pl{p} \cap \bd{k}{+} \pl{p} \), and since \( \pl{p} \) is round, \( x \in \bd{k - 1}{-} \pl{p} \cup \bd{k - 1}{+} \pl{p} \).
    Hence, using distributivity of \( \posplex \) with boundaries again, \( x \in \bd{k - 1}{-} \posplex(\pl{p}) \cup \bd{k - 1}{+} \posplex(\pl{p}) \).
    This concludes the induction and the proof.
\end{proof}

\begin{cor} \label{cor:oriented_poset_is_functor_strict_functor}
    The oriented graded poset of a polyplex defines a functor
    \begin{equation*}
        \posplex \colon \polyplexcat \to \molcat,
    \end{equation*}
    which is also a strict functor of the associated strict \( \omega \)\nbd categories.
    This functor sends plexes to atoms, round polyplexes to round molecules, and preserves dimensions.
\end{cor}

\begin{dfn}
    Let \( n \in \nat \cup \set{\omega} \). 
    We call again
    \begin{equation*}
        \posplex \colon \nPol n \to \dcpxn{n}
    \end{equation*}
    the Yoneda extension of \( \posplex \colon \nplexcat{n} \to \natomcat{n} \) given by Corollary \ref{cor:oriented_poset_is_functor_strict_functor}.
\end{dfn}

\noindent Recall that by \cite[2.60]{chanavat2025stricter}, the full subcategory inclusion \( \snCat n \incl \nCat n \) of stricter \( n \)\nbd categories into strict \( n \)\nbd categories has a left adjoint 
\begin{equation*}
    \rs \colon \nCat n \to \snCat n.
\end{equation*}

\begin{thm} \label{thm:from_polyplexes_to_molecules}
    Let \( n \in \nat \cup \set{\omega} \).
    Then the diagram
    \begin{center}
        \begin{tikzcd}
            {\nplexcat{n}} & {\natomcat n} \\
            {\nPol n} & {\dcpxn n} \\
            {\nCat n} & {\snCat n}
            \arrow["\posplex", from=1-1, to=1-2]
            \arrow[hook, from=1-1, to=2-1]
            \arrow[hook, from=1-2, to=2-2]
            \arrow["\posplex", from=2-1, to=2-2]
            \arrow[hook, from=2-1, to=3-1]
            \arrow["{\molecin{-}}", hook, from=2-2, to=3-2]
            \arrow["\rs", from=3-1, to=3-2]
        \end{tikzcd}
    \end{center}
    commutes up to natural isomorphisms, and, if \( n \le 3 \), the horizontal functors are equivalences of categories.
\end{thm}
\begin{proof}
    The upper square commutes by definition.
    Then, for any regular polygraph \( C \), \( \rs C \) is a stricter polygraph (in the sense of \cite[Definition 2.35]{chanavat2025stricter}) which is computed by the same cellular extensions as the one giving \( C \), but in \( \snCat{n} \) instead.
    Therefore, by Lemma \ref{lem:cellular_colimits_are_omega_colimits}, \( \rs C \) is naturally isomorphic to \( \molecin{\posplex(C)} \). 
    
    Now suppose that \( n \le 3 \).
    By \cite[Theorem 2.66]{chanavat2025stricter}, \( \rs \) is an equivalence of categories.
    Let \( U \) be a molecule of dimension \( \le n \).
    We show by induction on the submolecules \( V \) of \( U \) that \( \molecin{V} \) is a regular polyplex.
    When \( V \submol U \) is of type \( \pt \submol U \) this is clear.
    If \( U \) is an atom, the inductive hypothesis together with \cite[Proposition 2.3.3]{henry2018nonunital} imply that \( \molecin{U} \) is a regular polygraph and a plex. 
    Finally, suppose that \( U = V \cp{k} W \) with \( V \) and \( W \) proper submolecules.
    Then, since \( n \le 3 \), the comparison \( \molecin{V} \cup \molecin{W} \to \molecin{U} \) is an isomorphism by \cite[Theorem 2.66]{chanavat2025stricter}.
    Since the domain of that map is a regular polyplex by inductive hypothesis, we conclude that \( \molecin{U} \) is a regular polyplex.  
    This shows that \( \molecin{-} \colon \natomcat{n} \to \nplexcat n \) is a well-defined functor, and it is straightforward to see that \( \molecin{-} \) and \( \posplex \) are inverses of each other. 
    Thus \( \posplex \colon \nplexcat{n} \to\natomcat{n} \) is an equivalence of categories, hence so is its Yoneda extension \( \posplex \).
    This concludes the proof.
\end{proof}

\begin{rmk} 
    It is still the case that \( \posplex \colon \nPol 4 \to \dcpxn{4} \) is an equivalence of categories.
    Indeed, \( \npolyplexcat{3} \cong \nmolcat 3 \) implies by a similar argument as in the previous proof that the categories \( \nplexcat 4 \) and \( \natomcat 4 \) are equivalent via \( \posplex \) and \( \molecin{-} \), but this does \emph{not} produce an equivalence between \( \npolyplexcat{4} \) and \( \nmolcat{4} \).
    Similarly, \( \rs \colon \nCat{4} \to \snCat{4} \) is not an equivalence of categories anymore.
\end{rmk}
\section{Poset of subdivisions}
\label{sec:subdiv}

\subsection{Layering of molecules}

\begin{dfn} [Pre-layering and layering] \label{dfn:prelayering}
    Let \( U \) be a molecule and \( k \geq -1 \).
    A \emph{\( k \)\nbd pre-layering of \( U \)} comprises a number \( m \geq 1 \) and a sequence \( (\order{i}{U})_{i = 1}^m \) of submolecules of \( U \) of dimension \( > k \) such that \( U \) is isomorphic to
    \begin{equation*}
        \order{1}{U} \cp{k} \ldots \cp{k} \order{m}{U}.
    \end{equation*}
    If \( k = -1 \), then it is implied that \( m = 1 \) and \( \order{1}{U} = U \).
    A \emph{\( k \)\nbd layering of \( U \)} is a \( k \)\nbd pre-layering such that 
    \begin{equation*}
        m = \left| \bigcup_{i > k} \gr{i}{(\maxel{U})} \right|.
    \end{equation*} 
\end{dfn}

\begin{rmk}
    The definition of \( k \)\nbd layering coincides with \cite[4.2.1]{hadzihasanovic2024combinatorics}.
\end{rmk}

\begin{exm}
    Any split \( A \cp{k} B \) of a molecule \( U \) into proper submolecules \( A \) and \( B \) is a \( k \)\nbd pre-layering.
\end{exm}

\begin{dfn} [Poset of pre-layering] \label{dfn:poset_of_pre_layering}
    Let \( U \) be a molecule, \( k \geq -1 \), and \( L \eqdef (\order{i}{U})_{i = 1}^m \), \( K \eqdef (\order{i}{V})_{i = 1}^p \) be two \( k \)\nbd pre-layerings. 
    We say that \emph{\( L \) refines \( K \)} if there exists an endpoint and order-preserving injection 
    \begin{equation*}
        \fun{f} \colon \set{0 < \ldots < p} \to \set{0 < \ldots < m}
    \end{equation*}
    such that for all \( i \in \set{1, \ldots, p} \), we have
    \begin{equation*}
        \order{i}{V} = \order{\fun{f}(i - 1) + 1}{U} \cp{k} \ldots \cp{k} \order{\fun{f}(i)}{U}.
    \end{equation*}
    We write \( \preLay k U \) for the poset of \( k \)\nbd pre-layerings of \( U \) up to layer-wise isomorphism, ordered by letting \( K \le L \) if \( L \) refines \( K \).
    We also let \( \Lay k U \) be its subposet on \( k \)\nbd layerings.
\end{dfn}

\begin{rmk}
    Since the only \( k \)\nbd pre-layering refining a \( k \)\nbd layering is itself, \( \Lay k U \) is in fact just the set of maximal elements of \( \preLay{k} U \).
\end{rmk}

\begin{dfn} [Linearly ordered partition]
    Let \( X \) be a set.
    A \emph{linearly ordered partition of \( X \)} comprises a partition of \( X \) together with a linear ordering \( E \eqdef (\order{i}{E})_{i = 1}^m \) of the equivalence classes.
    We write \( \prt{E} \colon X \to \set{1, \ldots, m} \) for the unique surjective function such that \( \invrs{\prt{E}}(i) = \order{i}{E} \) for all \( i \in \set{1, \ldots, m} \).
    If \( \prt{E} \) is injective, we may also write \( (\order{i}{x})_{i = 1}^m \) in place of \( (\set{\order{i}{x}})_{i = 1}^m \).
\end{dfn}

\begin{dfn}
    Let \( U \) be an oriented graded poset and \( x \in U \).
    For any \( k \in \nat \) and \( \a \in \set{-, +} \), we write 
    \begin{equation*}
       \faces{k}{\a} x \eqdef \set{y \in \clset{x} \mid \cofaces{}{-\a} y \cap \clset{x} = \varnothing}, 
    \end{equation*}
    where \( \cofaces{}{-\a} y \) is the set of elements \( z \in U \) such that \( y \in \faces{}{-\a} z \).
\end{dfn}

\begin{dfn} [Maximal flow graph] \label{dfn:maxflow}
    Let \( U \) be a molecule and \( k \geq -1 \).
    The \emph{maximal \( k \)\nbd flow graph on \( U \)} is the (directed) graph \( \maxflow k U \) whose vertices are given by the set \( \bigcup_{i > k} \gr{i}{(\maxel{U})} \), and with an edge from \( x \) to \( y \) if \( \faces{k}{+} x \cap \faces{k}{-} y \neq \varnothing \).
\end{dfn}

\begin{comm} \label{comm:shift_index_preorder}
    In \cite{campion2023pasting}, Campion uses a preorder \( \mathrm{At}_k(U) \) associated to any torsion-free complex \( U \).
    We will see in Proposition \ref{prop:frame_acyclic_iso_preorderings} that the right alternative for us is (the preorder generated by) \( \maxflow{k - 1}{U} \); we point out the shift in the index.
\end{comm}

\begin{dfn} [Pre-ordering and ordering] \label{dfn:preordering}
    Let \( U \) be a molecule and \( k \geq 1 \).
    An \emph{\( k \)\nbd pre-ordering of \( U \)} is a linearly ordered partition \( L \eqdef (\order{i}{L})_{i = 1}^m \) of the vertices of \( \maxflow{k}{U} \) such that, for all edges \( x \to y \) of \( \maxflow{k}{U} \), we have \( \prt{L}(x) \le \prt{L}(y) \).
    A \emph{\( k \)\nbd ordering of \( U \)} is a \( k \)\nbd pre-ordering which is a linear order.
    We write \( \preOrd k U \) for the poset of \( k \)\nbd pre-orderings of \( U \) ordered by refinement, and \( \Ord k U \) for its subposet on \( k \)\nbd orderings.
\end{dfn}

\begin{rmk}
    Since partition associated to \( k \)\nbd orderings are discrete, \( \Ord{k} U \), which coincides with \cite[Definition 4.3.10]{hadzihasanovic2024combinatorics}, is in fact just the set of maximal elements of \( \preOrd{k} \).
\end{rmk}

\begin{rmk} \label{rmk:notation_lin_preord}
    Let \( X \) be the preorder generated by the edges of \( \maxflow{k}{U} \).
    Then \( \preOrd k U \) coincides, as a poset, with \( \mathrm{Lin}^\triangleleft(X) \) in the notation of \cite[Definition 3.1]{campion2023pasting}.
\end{rmk}

\begin{prop} \label{prop:preorderings_from_layerings}
    Let \( U \) be a molecule, \( k \geq - 1 \).
    For each \( k \)\nbd pre-layering \( L \eqdef (\order{i}{U})_{i = 1}^m \) of \( U \), write \( (\order{i}{L})_{i = 1}^m \) for the linearly ordered partition on the vertices of \( \maxflow{k}{U} \) defined by letting \( \order{i}{L} \) be the collection of vertices of \( \maxflow{k}{U} \) that belong to \( \order{i}{U} \).
    Then the assignment
    \begin{equation*}
        \o k U \colon L \mapsto\; (\order{i}{L})_{i = 1}^m
    \end{equation*}
    determines an injective and order-preserving function \( \preLay{k} U \incl \preOrd{k} U \), which restricts to an injective function \( \Lay{k} U \incl \Ord{k} U \) on layerings.
\end{prop}
\begin{proof}
    By a straightforward variation of \cite[Proposition 4.3.11]{hadzihasanovic2024combinatorics}, we conclude that the assignment is well-defined and determines an injective function.
    It remains to show that it is order-preserving.
    Suppose that we have two \( k \)\nbd pre-layering \( L \eqdef (\order{i}{U})_{i = 1}^p \) and \( K \eqdef (\order{i}{V})_{i = 1}^m \) such that \( L \) refines \( K \).
    Then letting \( \fun{f} \colon \set{0 < \ldots < p} \to \set{0 < \ldots < m} \) be the ordered-preserving injection as in Definition \ref{dfn:poset_of_pre_layering}, we see that, for all \( i \in \set{1, \ldots, m} \), we have
    \begin{equation*}
        \order{i}{K} = \order{\fun{f}(i - 1) + 1}{L} \cup \ldots \cup \order{\fun{f}(i)}{L},
    \end{equation*}
    so that \( \o k U L \) refines \( \o k U K \).
    This shows that \( \o k U \) is order-preserving and concludes the proof.
\end{proof}

\begin{dfn} [Frame-dimension]
    Let \( U \) be a molecule.
    The \emph{frame dimension of \( U \)} is the integer
    \begin{equation*}
        \frdim U \eqdef \dim \bigcup \set{\clset{x} \cap \clset{y} \mid x, y \in \maxel{U}, x \neq y}.
    \end{equation*}
\end{dfn}

\begin{dfn} [Frame-acyclic molecules] \label{dfn:frame_acyclic}
    Let \( U \) be a molecule.
    We say that \( U \) is \emph{frame-acyclic} if for all submolecules \( V \submol U \) of \( U \), if \( r \eqdef \frdim U \), then \( \maxflow{r}{V} \) is acyclic.
\end{dfn}

\begin{lem} \label{lem:frame_acyclic_properties}
    Let \( U \) be a frame-acyclic molecule and let \( k \in \nat \) with \( \frdim U \le k \le \dim U - 1 \).
    Then \( U \) admits a \( k \)\nbd layering and the function \( \o k U \colon \Lay k U \incl \Ord{k} U \) is bijective.
\end{lem}
\begin{proof}
    This is a particular case of \cite[Corollary 8.1.5]{hadzihasanovic2024combinatorics}.
\end{proof}

\begin{lem} \label{lem:pre_ordering_are_refined_in_frame_acyclic}
    Let \( U \) be a frame-acyclic molecule and let \( k \in \nat \) with \( \frdim U \le k \le \dim U - 1 \).
    Then every \( k \)\nbd pre-ordering of \( U \) is refined by a \( k \)\nbd ordering.
\end{lem}
\begin{proof}
    Let \( (\order{i}{L})_{i = 1}^p \) be a \( k \)\nbd pre-ordering of \( U \).
    By \cite[Proposition 4.3.8]{hadzihasanovic2024combinatorics} and Lemma \ref{lem:frame_acyclic_properties}, \( \maxflow{k}{U} \) is acyclic.
    Therefore, for all \( i \in \set{1, \ldots, p} \), the induced subgraphs of \( \maxflow{k}{U} \) on the vertices \( \order{i}{L} \) are acyclic, hence admit an ordering \( (\order{i,j}{x})_{j = 1}^{m_i} \).
    Assembling all these orderings together gives a \( k \)\nbd ordering \( ((\order{i, j}{x})_{j = 1}^{m_i})_{i = 1}^p \) of \( U \) which refines \( (\order{i}{L})_{i = 1}^p \).
    This concludes the proof.
\end{proof}

\noindent The next Proposition is our alternative for \cite[Proposition 4.15]{campion2023pasting}.

\begin{prop} \label{prop:frame_acyclic_iso_preorderings}
    Let \( U \) be a frame-acyclic molecule and \( k \in \nat \) with \( \frdim U \le k \le \dim U - 1 \).
    Then the function \( \o k U \colon \preLay{k} U \incl \preOrd{k} U \) is an isomorphism of posets.
\end{prop}
\begin{proof}
    We already know by Proposition \ref{prop:preorderings_from_layerings} that the order preserving function is injective.
    Conversely, consider a \( k \)\nbd pre-ordering \( (\order{i}{L})_{i = 1}^p \) of \( U \).
    By Lemma \ref{lem:pre_ordering_are_refined_in_frame_acyclic}, this \( k \)\nbd pre-ordering is refined by a \( k \)\nbd ordering \( (\order{i}{x})_{i = 1}^m \), and by Lemma \ref{lem:frame_acyclic_properties}, we know that \( \o k U \colon \Lay{k} U \incl \Ord{k} U \) is a bijection.
    Thus \( (\order{i}{x})_{i = 1}^m \) induces a \( k \)\nbd layering \( (\order{i}{U})_{i = 1}^m \) of \( U \).

    Since \( (\order{i}{L})_{i = 1}^p \) is refined by \( (\order{i}{x})_{i = 1}^m \) letting \( j_{i - 1} \eqdef \min \set{j \mid \order{j}{x} \in \order{i}{L}} \) for \( i \in \set{1, \ldots, p} \), we have the \( k \)\nbd pre-layering
    \begin{equation*}
        (\order{j_{i - 1}}{U} \cp{k} \ldots \cp{k} \order{j_i - 1}{U})_{i = 1}^p
    \end{equation*} 
    of \( U \), whose image via \( \o k U \) is \( (\order{i}{L})_{i = 1}^p \).
    This shows that \( \o k U \) is bijective, whose inverse is order-preserving by construction.
\end{proof}

\subsection{Generalities on thetas}

\begin{dfn} [\( k \)\nbd globe]
    Let \( k \in \nat \).
    The \emph{\( k \)\nbd globe}, written \( \globe{k} \), is defined inductively by \( \globe{0} \eqdef \pt \), and for \( k > 0 \), by \( \globe{k} \eqdef \globe{k - 1} \celto \globe{k - 1} \).
\end{dfn}

\noindent The following definitions are adapted from \cite[Definition 4.2, Observation 4.6]{campion2023pasting}, now using the notations relative to molecules. 

\begin{dfn} [\( S \)\nbd theta]
    Let \( S \subseteq \nat \cup \set{-1} \).
    The class of \emph{\( S \)\nbd thetas} is the inductive class of molecules generated by the following two clauses:
    \begin{enumerate}
        \item for all \( k \in \nat \), the \( k \)\nbd globe \( \globe{k} \) is an \( S \)\nbd theta.
        \item for all \( S \)\nbd thetas \( U \) and \( V \), and \( k \in S \), if \( U \cp{k} V \) is defined, it is an \( S \)\nbd theta.
    \end{enumerate}
    We write \( \ThetaS S \) for the full subcategory of \( \dcpxomega \) on objects \( \molecin{U} \), for \( U \) an \( S \)\nbd theta, and, for \( n \in \nat \), \( \ThetanS n S \) for its full subcategory on \( S \)\nbd thetas of dimension \( \le n \).
    If \( S = \nat \), we drop the prefix \( S \), and simply speak of thetas, and write \( \Theta \) and \( \Thetan{n} \) for the associated categories.
    We may also write \( \Thetan{\omega} \eqdef \Theta \).
\end{dfn}

\begin{rmk}
    By \cite[Lemma 9.1.2 and Theorem 11.2.18]{hadzihasanovic2024combinatorics}, \( \Theta \) is isomorphic to the usual Joyal's theta category \cite{joyal1997disk}.
\end{rmk}

\begin{rmk}
    Since the operations \( - \cp{k} - \) are only defined for \( k \geq 0 \), these definitions imply that \( \ThetaS \varnothing = \ThetaS {\set{- 1 }} \) and contains exactly the globes.
\end{rmk}

\begin{comm}
    We will often identify a theta \( \theta \) with the unique, up to unique isomorphism, molecule \( U_\theta \) such that \( \molecin{U_\theta} = \theta \).
\end{comm}

\begin{exm}
    The reader familiar with the inductive definition of \( \Thetan{n} \) by iterated wreath product knows that, if \( n > 0 \), the objects of \( \Theta \wr \Thetan{n - 1} \) are given by tuples \( (k; \theta_1, \ldots, \theta_k) \), where \( k \in \nat \) and the \( \theta_i \)'s are objects of \( \Thetan{n - 1} \).
    Then, the object of \( \Thetan{n} \) represented by this tuple is the molecule given by the pasting
    \begin{equation*}
        \sus{\theta_1} \cp{0} \ldots \cp{0} \sus{\theta_k},
    \end{equation*}
    where \( \sus{} \) is the suspension of molecules, which coincides with the usual \( \omega \)\nbd categorical suspension, see \cite[Definition 7.3.1 and Proposition 7.3.22]{hadzihasanovic2024combinatorics}.
\end{exm}

\begin{dfn}
    Let \( i, j \in \nat \cup \set{-1, \omega} \).
    We write \( \ivl{i}{j} \) for the elements \( k \) of \( \nat \cup \set{-1, \omega} \) such that \( i \le k \le j \).
\end{dfn}

\begin{dfn} [Canonical pre-layering]
    Let \( n, k \in \nat \) such that \( 0 \le k \le n \), and let \( \theta \) be an \( \ivl{k}{n} \)\nbd theta.
    Then there exist unique \( m \geq 1 \) and \( k \)\nbd pre-layering \( (\order{i}{\theta})_{i = 1}^m \) of \( \theta \) such that for all \( i \in \set{1, \ldots, m} \), \( \order{i}{\theta} \) is a \( \ivl{k + 1}{n} \)\nbd theta.
    This decomposition is called the \emph{canonical \( k \)\nbd pre-layering of \( \theta \)}.
\end{dfn}

\begin{dfn} \label{dfn:adjunction_thetaS_thetaT}
    Let \( S \subseteq T \subseteq \nat\) be two subsets of natural numbers.
    Since any \( S \)\nbd theta is in particular a \( T \)\nbd theta, \( \ThetanS{n}{T} \) is a full subcategory of \( \ThetanS{n}{S} \).
    By \cite[Observation 4.6]{campion2023pasting}, when \( S \) is an initial segment of \( T \), the full subcategory inclusion has a right adjoint 
    \begin{equation*}
        \Rtheta \colon \ThetanS{n}{T} \to \ThetanS{n}{S}
    \end{equation*}
    that we define from \( \ThetanS{n}{T \cap [k, n]} \) by backward induction on \( k \).
    When \( k = n + 1 \), then \( \Rtheta \) sends globes to globes.
    Inductively, it sends the canonical \( k \)\nbd pre-layering \( (\order{i}{\theta})_{i = 1}^m \) of a theta \( \theta \) to
    \begin{equation*}
        \Rtheta(\order{1}{\theta}) \cp{r} \ldots \cp{r} \Rtheta(\order{m}{\theta})
    \end{equation*} 
    if \( r \in S \), or to \( \globe{\dim \theta} \) if \( r \in T \setminus S \).
    
    Since there is a unique active strict functor \( \globe{\dim \theta} \to \theta \), this construction extends to an active monomorphism \( \Rtheta(\theta) \to \theta \), giving the counit of the adjunction.
\end{dfn}

\begin{dfn} \label{dfn:strict_n_cat_are_presheaf_theta}
    Given a natural number \( n \in \nat \), we write \( \Psh(\Thetan{n}) \) for the category of (set-valued) presheaves over \( \Thetan{n} \).
    We recall that the full subcategory inclusion \( \Thetan{n} \incl \nCat{n} \) induces a nerve functor \( \N  \colon \nCat{n} \to \Psh(\Thetan{n}) \), which is fully faithful \cite{berger2002cellular}.
    From now on, we always identify a small strict \( n \)\nbd category \( C \) with its associated nerve \( \N C \). 
\end{dfn}

\begin{exm}
    Let \( X \) be a directed \( n \)\nbd complex. 
    Then \( \molecin{X} \) is the presheaf over \( \Thetan{n} \) whose sections of shape \( \theta \) are strict functors \( u \colon \theta \to \molecin{X} \).
    In particular, we may import terminology that applies to strict functors of molecules to the sections of \( \molecin{X} \).
\end{exm}

\noindent The following is well-known, we only reproduce the proof to showcase how Proposition \ref{prop:ofs_functor_of_dcpx} specialises to the theta category.  
\begin{prop} \label{prop:ternary_factorisation_of_theta}
    Let \( n \in \nat \).
    The classes of collapses, active monomorphisms, and embeddings form a ternary factorisation system on \( \Thetan{n} \).
\end{prop}
\begin{proof}
   By Corollary \ref{cor:ofs_functor_of_rdcpx}, every morphism \( u \colon \theta \to \theta' \) factors uniquely as an active strict functor followed by a local embedding.
   By \cite[Proposition 8.3.23 and Lemma 9.1.2]{hadzihasanovic2024combinatorics}, the local embedding is an embedding.
   Since any embedding of a molecule in a theta is again a theta, this gives a first factorisation system \( (L, R) \), where \( R \) is the class of embeddings.
   By \cite{bergner2013reedy}, there is another factorisation system \( (L', R') \) on \( \Thetan{n} \), whose right class are the monomorphisms.
   Since any embedding is a monomorphism, we get a ternary factorisation system \( (L', L \cap R', R) \).
   It only remains to show that the class \( L' \) is that of collapses between thetas.
   Since a strict functor \( \F \colon \molecin{U} \to \molecin{V} \) is a collapses if and only if it restricts to a surjective function from the objects of \( \cell U \) to the ones of \( \cell V \), we may inductively conclude using the description of \( L' \) given in \cite[2.7]{bergner2013reedy}.
\end{proof}

\begin{rmk} \label{rmk:theta_ez_cat}
    The classes of collapses and monomorphisms form a so-called \emph{Eilenberg-Zilber category}, see \cite[Corollary 4.4]{bergner2013reedy}. 
\end{rmk}

\begin{dfn} [Non-degenerate section]
    Let \( X \) be a presheaf over \( \Thetan{n} \), and \( u \colon \theta \to X \) be a section of \( X \).
    We say that \( u \) is \emph{degenerate} if there exists a non-identity collapse \( f \colon \theta \to \theta' \) and a section \( v \colon \theta' \to X \) such that \( u = v \after f \).
    Otherwise, we say that \( u \) is \emph{non-degenerate}.
    By Remark \ref{rmk:theta_ez_cat}, \( u \) can be uniquely decomposed as \( v \after f \), where \( f \) is a collapse and \( v \) is non-degenerate.
\end{dfn}

\begin{lem} \label{lem:submolecule_of_thetas}
    Let \( U \) be a theta and \( V \submol U \) a submolecule.
    Then \( V \) is a theta.
\end{lem}
\begin{proof}
    By \cite[Proposition 9.1.28]{hadzihasanovic2024combinatorics}, a molecule \( W \) is a theta if and only if for all \( x \in W \), \( \clset{x} \) is a globe, which concludes the proof.
\end{proof}

\begin{lem} \label{lem:active_non_degenerate_theta_iff}
    Let \( U \) be a molecule, and \( u \colon \theta \to \molecin{U} \) be an active strict functor.
    The following are equivalent.
    \begin{enumerate}
        \item \( u \) is non-degenerate;
        \item for all \( x, y \in \theta \), if \( u(\imel{\theta}{x}) = u(\imel{\theta}{y}) \), then \( x = y \);
        \item for all \( x \in \theta \), \( \dim u(\imel{\theta}{x}) = \dim x \).
    \end{enumerate}
\end{lem}
\begin{proof}
    Suppose first that \( u \) is degenerate.
    Then \( u = \molecin{f} \after v \), for \( f \colon \theta \to \theta' \) a non-identity collapse between thetas, and \( v \colon \theta' \to \molecin{U} \) a non-degenerate strict functor. 
    Since \( f \) is a non-identity surjection, there exists \( x, y \in \theta \) with \( x \neq y \) such that \( f(x) = f(y) \).
    Now \cite[Proposition 6.2.19]{hadzihasanovic2024combinatorics} and Remark \ref{rmk:factorisation_collapse_local_embedding} imply that \( f \) is not dimension-preserving, thus there exists \( x \in \theta \) such that \( \dim f(x) < \dim x \).
    This proves one direction.

    Conversely, suppose \( u \) is non-degenerate.
    Suppose that there exist \( x \in \theta \) such that \( \dim u(\imel{\theta}{x}) = k < \dim x \). 
    Then for all \( y \geq x \), also \( \dim u(\imel{\theta}{x}) < \dim y \). 
    Now pick any \( y \geq x \) maximal in \( \theta \), and let \( \theta' \) be the theta obtained by collapsing \( \globe{\dim y} \cong \clset{y} \subseteq \theta \) to \( \globe{\dim y - 1} \).
    Then there is a collapse \( f \colon \theta \to \theta' \) which is not the identity, and a cell \( v \colon \theta' \to \molecin{U} \) such that \( v \after \molecin{f} = u \), contradicting that \( u \) is non-degenerate.
    This shows that the first point and the third point are equivalent.

    In particular, if \( \iota \colon \theta' \submol \theta \) is a submolecule then \( \theta' \) is a theta by Lemma \ref{lem:submolecule_of_thetas} and \( \restr{u}{\theta'} \) is active and non-degenerate.
    Let \( n \in \nat \).
    We show by backward induction on \( 0 \le k \le n \) that the statement holds for all active strict functors \( u \colon \theta \to \molecin{U} \) with \( \theta \) an \( \ivl{k}{n - 1} \)\nbd theta.
    The base case is whenever \( \theta \) is a globe.
    By the third point, we have \( k \eqdef \dim x = \dim y \) and \( \dim U = \dim \theta \).
    If \( k = \dim \theta \), then since \( \theta \) has a unique maximal element, \( x = y \).
    Thus we may suppose \( k < \dim \theta \).
    Then \( \mapel{x} \) is an embedding of the form \( \bd{k}{\a} \theta \incl \theta \) and \( \mapel{y} \) is an embedding of the form \( \bd{k}{\beta} \theta \) for some \( \a, \beta \in \set{-, +} \).
    Since \( u \) is active, this means that \( \bd{k}{\a} U = \bd{k}{\beta} U \).
    If \( \a \neq \beta \), then \( \dim U \le k \), but \( \dim U = \dim \theta > k \).
    Hence \( \a = \beta \), so that \( x = y \).

    Inductively, let \( 0 \le k < n \), and suppose that \( u \colon \theta \to \molecin{U} \) is an active strict functor such that \( \theta \) has \( \order{1}{\theta} \cp{k} \ldots \cp{k} \order{m}{\theta} \) for canonical \( k \)\nbd pre-layering, consider \( x, y \in \theta \) with \( u(\imel{\theta}{x}) = u(\imel{\theta}{y}) \). 
    Let \( i \) and \( j \) be respectively maximal and minimal such that \( x, y \in \order{i}{\theta} \cp{k} \ldots \order{j}{\theta} \), and let \( u' \colon \theta' \to \molecin{U'} \) be the restriction of \( u \) to \( \order{i}{\theta} \cp{k} \ldots \cp{k} \order{j}{\theta} \), which is still non-degenerate and active.
    We wish to show that \( i = j \).
    By means of contradiction, suppose that \( i < j \).
    Then \( u(\imel{\theta}{x}) \subseteq u(\order{i}{\theta}) \cap u(\order{j}{\theta}) \) has dimension \( \le k \), and since \( \restr{u}{\imel{\theta}{x}} \) is again non-degenerate, we obtain \( \ell \eqdef \dim y = \dim x \le k \).
    Since the elements of dimension \( \le k \) in \( \theta' \) are contained in \( \bd{k}{} \theta' \) together with the union of the intersections \( \order{r}{\theta} \cap \order{r + 1}{\theta} \) for all \( r \in \set{1, \ldots, m - 1} \), and since \( i \) and \( j \) have been chosen maximal and minimal respectively, we have \( \a \in \set{-, +} \) such that \( \imel{\theta}{x} \subseteq \bd{k}{\a} \theta' \) and \( \imel{\theta}{y} \subseteq \bd{k}{\beta} \theta' \).
    Noticing that for all \( \gamma \in \set{-, +} \), \( \bd{k}{\gamma} \theta' \) is an \( \varnothing \)\nbd theta, thus a globe, we see that the previous two embeddings are in fact equalities.
    Therefore, \( \imel{\theta}{x} = \bd{\ell}{\a} \theta' \) and \( \imel{\theta}{y} = \bd{\ell}{-\a} \theta' \).
    Finally, \( u(\imel{\theta}{x}) = u(\imel{\theta}{y}) \) implies that \( \bd{\ell}{-} U' = \bd{\ell}{+} U' \), thus \( \dim U' \le \ell \le k \), but since \( i < j \), we also have \( \dim U' = \dim \theta' > k \), a contradiction.
    Therefore, \( i = j \) and \( x, y \in \order{i}{\theta} \), so we may conclude the proof by inductive hypothesis.
\end{proof}

\begin{prop} \label{prop:quotient-free_non_degenerate_are_mono}
    Let \( U \) be a regular directed complex and \( u \colon \theta \to \molecin{U} \) be a non-degenerate quotient-free strict functor.
    Then \( u \) is a monomorphism.
\end{prop}
\begin{proof}
    Since any quotient-free strict functor factors, by Corollary \ref{cor:ofs_functor_of_rdcpx}, as an active strict functor followed by an embedding, which is a monomorphism, we may in fact suppose that \( u \) is active, hence that \( U \) is a molecule. 
    It is enough to show that for all pairs of morphisms \( v, w \colon \theta' \to \theta \) in \( \Theta \), if \( u \after v = u \after w \), then \( v = w \).
    Suppose given such a pair \( v, w \), and let \( x \in \theta' \).
    Then by Proposition \ref{lem:active_functor_preserve_submolecule} together with Lemma \ref{lem:submolecule_of_thetas}, \( v(\imel{\theta'}{x}) \) and \( w(\imel{\theta'}{x}) \) are two submolecules of \( \theta \) that are equalized by \( u \).
    By the second point of Lemma \ref{lem:active_non_degenerate_theta_iff}, they are the same submolecule.
    This shows that \( v \) and \( w \) agree on all the globes of \( \theta' \), thus \( v = w \).
    This concludes the proof.
\end{proof}

\begin{comm} \label{comm:alternative_non_degenerate}
    The previous Proposition is our alternative to \cite[Lemma 1.9]{campion2023pasting}: it is not the non-degenerate cells \( u \colon \theta \to \molecin{U} \) that are monomorphisms, but only the \emph{quotient-free} non-degenerate cells.
    Indeed, the pasting theorem we aim for contains, in general, many acyclic shapes.
    For instance, consider the following regular directed complex \( U \)
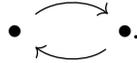
    \begin{center}
        \begin{tikzcd}
            \bullet & \bullet.
            \arrow[curve={height=-12pt}, from=1-1, to=1-2]
            \arrow[curve={height=-12pt}, from=1-2, to=1-1]
        \end{tikzcd}
    \end{center}
    Then there exists two local embeddings of type \( \arr \cp{0} \arr \to U \) that classify two non-degenerate cells of \( \molecin{U} \), yet none of them are monomorphisms.
\end{comm}

\begin{cor} \label{cor:strict_functor_out_of_theta}
    Let \( u \colon \theta \to \molecin{X} \) be a morphism in \( \dcpxomega \) whose domain is a theta.
    Then \( u \) factors uniquely, up to unique isomorphism, as a collapse, followed by an active monomorphism, followed by a pasting diagram.
\end{cor}
\begin{proof}
    Since \( \Theta \) is an Eilenberg-Zilber category, the cell \( u \) factors uniquely as a collapse followed by a non-degenerate cell \( v \colon \theta' \to \molecin{X} \).
    By Proposition \ref{prop:ofs_functor_of_dcpx}, \( v \) factors uniquely, up to unique isomorphism, as an active strict functor \( w \colon \theta' \to \molecin{U} \), followed by a diagram.
    Since \( w \) is active and \( \theta' \) is a molecule, \( U \) is a molecule, thus the diagram is a pasting diagram.
    It remains to show that \( w \) is a monomorphism, but this follows directly from Proposition \ref{prop:quotient-free_non_degenerate_are_mono} since \( w \) is non-degenerate, otherwise \( v \) would be. 
    This concludes the proof.
\end{proof}

\subsection{Contractibility of the poset of subdivisions}

In this section, we fix \( n \in \nat \).

\begin{dfn} [Big cell]
    Let \( U \) be a molecule.
    The \emph{big cell of \( U \)} is the unique active strict functor \( \bigcell U \colon \globe{\dim U} \to \molecin{U} \).
\end{dfn}

\begin{rmk} \label{rmk:big_cell_properties}
    By Lemma \ref{lem:active_non_degenerate_theta_iff}, the big cell of \( U \) is non-degenerate.
    Furthermore, for all active strict functors \( \F \colon \molecin{U} \to \molecin{V} \), \( \F \after \bigcell U = \bigcell V \).
\end{rmk}

\noindent We now introduce the same posets as \cite[Definition 4.1, Definition 4.7]{campion2023pasting}, and follow the same strategy as Campion to prove that they are contractible.

\begin{dfn} [Poset of subdivisions of a molecule] \label{dfn:poset_subdiv}
    Let \( U \) be a \( n \)\nbd molecule and let \( S \subseteq \set{0, \ldots, n} \).
    The \emph{initial poset of \( S \)\nbd subdivisions of \( U \)} is the full subcategory \( \SdiS S U \) of the slice \( \slice{\ThetanS{n}{S}}{\left(\molecin{U}\right)} \) on non-degenerate active strict functors.
    Since morphisms in \( \SdiS S U \) are necessarily active, Proposition \ref{prop:quotient-free_non_degenerate_are_mono} and Remark \ref{rmk:big_cell_properties} imply that \( \SdiS S U \) is a poset with initial object \( \bigcell U \).
    The \emph{poset of \( S \)\nbd subdivisions of \( U \)} is the subposet \( \SdS S U \eqdef \SdiS S U \setminus \set{\bigcell U} \) of \( \SdiS S U \).
    If \( S = \set{0, \ldots, n} \), we drop the prefix \( S \) and accordingly write \( \Sdi U \) and \( \Sd U \). 
\end{dfn}

\begin{exm}
    If \( k \geq 1 \), then \( \Sdi k\globe{1} \) is the poset of active subobjects of the \( k \)\nbd simplex, which is isomorphic to the set of ordered partitions of \( \set{1 < \ldots < k} \) ordered by refinement.
\end{exm}

\noindent Our next goal is to show Theorem \ref{thm:poset_of_subdivision_contractible_or_empty} asserting that when \( U \) is a frame-acyclic molecule (Definition \ref{dfn:frame_acyclic}), then the poset \( \Sd U \) is empty or contractible, depending on whether \( U \) is an atom, or a proper molecule. 

\begin{dfn}
    Let \( U \) be a molecule.
    If \( T \) is a subset of \( \ivl{0}{n} \) and \( S \) is an initial segment of \( T \), then the adjunction
    \begin{center}
        \begin{tikzcd}
            {\ThetanS n S} & {\ThetanS n T}
            \arrow[""{name=0, anchor=center, inner sep=0}, curve={height=6pt}, hook, from=1-1, to=1-2]
            \arrow[""{name=1, anchor=center, inner sep=0}, "\Rtheta"', curve={height=6pt}, from=1-2, to=1-1]
            \arrow["\dashv"{anchor=center, rotate=90}, draw=none, from=0, to=1]
        \end{tikzcd}
    \end{center}
    of Definition \ref{dfn:adjunction_thetaS_thetaT}, whose counit \( \varepsilon \) is pointwise an active monomorphism, gives an adjunction  
    \begin{center}
        \begin{tikzcd}
            {\SdiS S U} & {\SdiS T U.}
            \arrow[""{name=0, anchor=center, inner sep=0}, curve={height=6pt}, hook, from=1-1, to=1-2]
            \arrow[""{name=1, anchor=center, inner sep=0}, "\Utheta"', curve={height=6pt}, from=1-2, to=1-1]
            \arrow["\dashv"{anchor=center, rotate=90}, draw=none, from=0, to=1]
        \end{tikzcd}
    \end{center}
    by sending \( u \colon \theta \to \molecin{U} \) to \( u \after \varepsilon_\theta \).
\end{dfn}

\begin{dfn}
    Let \( U \) be a molecule, \( k \geq -1 \) and \( L \eqdef (\order{i}{U})_{i = 1}^m \) be a \( k \)\nbd pre-layering of \( U \).
    We write \( \bigcell{L} \) for the element 
    \begin{equation*}
       \bigcell{\order{1}{U}} \cp{k} \ldots \cp{k} \bigcell{\order{m}{U}} \colon \globe{\dim \order{1}{U}} \cp{k} \ldots \cp{k} \globe{\dim \order{m}{U}} \to \molecin{U} 
    \end{equation*}
    of \( \Sdik k U \).
    The assignment \( L \mapsto \bigcell{L} \) extends to an isomorphism between the posets \( \Sdik{k} U \) and \( \preLay{k} U \).
    In the sequel, we may identify a \( k \)\nbd pre-layering \( L \) of \( U \) with its associated element \( \bigcell{L} \) of \( \Sdik{k} U \). 
\end{dfn}

\noindent The next result is our analogue for \cite[Lemma 4.18]{campion2023pasting}.

\begin{prop} \label{prop:frame_acyclic_theta_Sd_r_contractible}
    Let \( U \) be a frame-acyclic \( n \)\nbd molecule, let \( r \eqdef \frdim U \).
    Then the poset \( \Sdk{r} U \) is empty if \( r = -1 \), and contractible otherwise.
\end{prop}
\begin{proof}
    If \( r = -1 \), then \( U \) is an atom, and  \( \Sdik{r} U \) is a singleton, thus \( \Sdk{r} U \) is empty.
    Thus, suppose that \( r \geq 0 \). 
    By Proposition \ref{prop:frame_acyclic_iso_preorderings}, \( \Sdik{r} U \) is isomorphic to \( \preOrd{r} U \), thus \( \Sdk{r} U \) is isomorphic to \( \preOrd{r} U \setminus \set{\bot} \), where \( \bot \) is the smallest element of \( \preOrd{r} U \) corresponding to the trivial \( r \)\nbd pre-layering \( (U)_{i = 1}^1 \).
    Now the preorder generated by the edges of \( \maxflow{r}{U} \) is not an equivalence relation, for \( \maxflow{r}{U} \) is acyclic by \cite[Proposition 4.3.8]{hadzihasanovic2024combinatorics} and Lemma \ref{lem:frame_acyclic_properties}, and has at least one edge since by definition of the frame dimension and since \( U \) is not an atom.
    We may conclude by Remark \ref{rmk:notation_lin_preord} and \cite[Corollary 3.4]{campion2023pasting}.
\end{proof}

\begin{comm}
    Let \( U \) be a frame-acyclic \( n \)\nbd molecule with \( r \eqdef \frdim U \).
    If \( U \) is round and not an atom, then \( r = \dim U - 1 \), and \( U \) has non-trivial \( k \)\nbd pre-layerings \emph{only} for \( k = r \).
    Thus \( \Sd U = \Sdk{r} U \) is contractible by the previous Proposition.
    Of course, we cannot be content with Theorem \ref{thm:poset_of_subdivision_contractible_or_empty} for round molecules only, since the proof of the pasting theorem (Theorem \ref{thm:cellular_extension_of_frame_acyclic}) will proceed by induction on the submolecule of a molecule, which are, most of the time, not round.   
\end{comm}

\begin{comm} \label{comm:why_proof_cocart}
    Our next goal is to show that if \( U \) is frame-acyclic, then for any natural number \( k \) with \( \frdim U \le k \le \dim U - 1 \), \( \Utheta \colon \SdiS{\ivl{0}{k}} U \to \SdiS{\ivl{0}{k - 1}} U \) is a cocartesian fibration (Lemma \ref{lem:compose_is_cocartesian_one_step}).
    It turns out that we can follow the proof of \cite[Lemma 4.12]{campion2023pasting}, even if the statement is equivalent to the stronger \cite[Lemma 4.16]{campion2023pasting}.
    This is because our Proposition \ref{prop:frame_acyclic_iso_preorderings} is already at a right level of generality, which allows us to skip some extra work.
    We nonetheless thought it was safer to give an explicit proof in our context since we could not exclude that we did not miss some important details.
\end{comm}

\begin{dfn}
    Let \( U \) be a frame-acyclic \( n \)\nbd molecule, let \( k \in \nat \) be a natural number with \( \frdim U \le k \le n - 1 \), let \( L \) be a \( k \)\nbd pre-layering of \( U \), and let \( u \colon \theta \to \molecin{U} \) be an element of \( \SdiS{\ivl{0}{k}} U \).
    For each \( x \) in \( M \eqdef \gr{> k}{\maxel{\theta}} \), we let \( L_x \) be \( k \)\nbd pre-layering of the submolecule \( u(\imel{\theta}{x}) \) of \( U \) obtained by restricting \( L \), which is possible by Proposition \ref{prop:frame_acyclic_iso_preorderings}.
    We define \( \hat{u} \colon \hat{\theta} \to \molecin{U} \) to be the universal strict functor given by the following pushout in strict \( \omega \)\nbd categories
    \begin{center}
        \begin{tikzcd}
            {\bigcup\limits_{x \in M} \imel \theta x} & {\bigcup\limits_{x \in M} \xi_{L_x}} & \\
            \theta & {\hat\theta} & {\bigcup_x\molecin{u(\imel{\theta}{x})}} \\
            && {\molecin{U}.}
            \arrow[""{name=0, anchor=center, inner sep=0}, from=1-1, to=1-2]
            \arrow[hook', from=1-1, to=2-1]
            \arrow[hook', from=1-2, to=2-2]
            \arrow["{\bigcup_x \bigcell {L_x}}", curve={height=-12pt}, from=1-2, to=2-3]
            \arrow[from=2-1, to=2-2]
            \arrow["u"', curve={height=12pt}, from=2-1, to=3-3]
            \arrow["{\hat u}", from=2-2, to=3-3]
            \arrow[hook', from=2-3, to=3-3]
            \arrow["\lrcorner"{anchor=center, pos=0.125, rotate=180}, draw=none, from=2-2, to=0]
        \end{tikzcd}
    \end{center}
\end{dfn}

\begin{rmk}
    Since \( \theta \) is a \( \ivl{0}{k} \)\nbd theta, \( \hat{\theta} \) is also a \( \ivl{0}{k} \)\nbd theta, where we replaced the maximal globes of dimension \( > k \) with \( \set{k} \)\nbd thetas.
    Furthermore, if \( \theta \) is a \( \ivl{0}{k - 1} \)\nbd theta, the strict functor \( \theta \to \hat{\theta} \) is the counit \( \counit_{\hat{\theta}} \) of the right adjoint \( \Utheta \colon  \SdiS{\ivl{0}{k}} U \to \SdiS{\ivl{0}{k - 1}} U \).
\end{rmk}

\begin{lem} \label{lem:properties_cocartesian_lift}
    Let \( U \) be a frame-acyclic \( n \)\nbd molecule, let \( k \in \nat \) be a natural number with \( \frdim U \le k \le \dim U - 1 \), let \( L \) be a \( k \)\nbd pre-layering of \( U \).
    Then
    \begin{enumerate}
        \item for all \( u, v \in \SdiS{\ivl{0}{k}} U \), if \( u \le v \), then \( \hat{u} \le \hat{v} \).
        \item for all \( u \in \SdiS{\ivl{0}{k}} U \), if \( \bigcell{L} \le u \), then \( \hat{u} = u \).
    \end{enumerate}
\end{lem}
\begin{proof}
    For the first point, let \( u \colon \theta \to \molecin{U} \), \( v \colon \theta' \to \molecin{U} \) and \( \F \colon \theta \to \theta' \) be active such that \( v \after \F = u \).
    Then, for each \( x \in \maxel{\theta} \) of dimension \( > k \), \( \F(\imel{\theta}{x}) \) is a \( \set{k} \)\nbd theta, so that letting \( M_x \eqdef \F(\imel{\theta}{x}) \cap \gr{> k}{\left(\maxel{\theta'}\right)} \), we have 
    \begin{equation*}
        \bigcell{L_x} \le \bigcup_{y \in M_x} \bigcell{L_y},
    \end{equation*}
    since \( \bigcup_{y \in M_x} \bigcell{L_y} \) is a \( k \)\nbd pre-layering of \( u(\imel{\theta}{x}) \) refining \( \bigcell{L_x} \).
    By the universal property of the pushout, this induces a strict functor \( \hat{\theta} \to \hat{\theta'} \) witnessing \( \hat{u} \le \hat{v} \).

    Next, let \( u \colon \theta \to \molecin{U} \) be in \( \SdiS{\ivl{0}{k}} U \) with \( \bigcell{L} \le u \), and let \( x \) be a maximal element of \( \theta \) with \( \dim x > k \). 
    Then, \( \o{k}{U} L_x \) is the codiscrete partition, since \( \bigcell{L} \le u \) implies that \( u(\imel{\theta}{x}) \) is fully contained in one of the layers of \( L \).
    Thus \( \bigcell{L_x} \colon \xi_{L_x} \to \molecin{u(\imel{\theta}{x})} \) equals the big cell \( \bigcell{u(\imel{\theta}{x})} \), so that the upper-horizontal strict functor in the pushout constructing \( \hat{u} \) is the identity.
    This shows that \( \hat{u} = u \) and concludes the proof.  

\end{proof}

\begin{lem} \label{lem:compose_is_cocartesian_one_step}
    Let \( k \in \nat \), and let \( U \) be a frame-acyclic \( n \)\nbd molecule with \( \frdim U \le k \le \dim U - 1 \).
    Then the functor \( \Utheta \colon \SdiS{\ivl{0}{k}} U \to \SdiS{\ivl{0}{k - 1}} U \) is a cocartesian fibration.
\end{lem}
\begin{proof}
    For each \( \theta \) in \( \ThetanS{n}{\ivl{0}{k}} \), let \( \counit_\theta \colon \Rtheta(\theta) \to \theta \) be the counit component of the adjunction between the inclusion \( \ThetanS{n}{\ivl{0}{k - 1}} \subseteq \ThetanS{n}{\ivl{0}{k}} \) and \( \Rtheta \).
    Recall that \( \Utheta \) is the functor sending \( u \colon \theta \to \molecin{U} \) to \( u \after \counit_\theta \).
    We show by backward induction on \( 0 \le r \le k \) the existence of a cartesian lift of \( \Utheta(u) \le v \), for \( u \in \SdiS{\ivl{r}{k}} U \) and \( v \in \SdiS{\ivl{0}{k - 1}} U \).
    
    Now consider the case \( r = k \).
    Then, since \( \theta \) is a \( \set{k} \)\nbd theta, we have \( u = \bigcell{L} \), for a \( k \)\nbd pre-layering \( L \).
    Then \( \Utheta(u) \) is the initial object of \( \SdiS{\ivl{0}{k - 1}} U \).
    For any \( v \in \SdiS{\ivl{0}{k - 1}} U \), \( \Utheta(u) \le v \) and, by Lemma \ref{lem:properties_cocartesian_lift}, we have \( u = \widehat{\Utheta(u)} \le \hat{v} \), and we claim this is a cocartesian lift of \( \Utheta(u) \le v \).
    Indeed, if \( u \le u' \) and \( v \le \Utheta(u') \), then since \( \Utheta(u') \le u' \) by \( \counit \), the same Lemma gives \( \hat{v} \le \hat{u}' = u' \). 
    This shows that \( u \le \hat{v} \) is a cocartesian lift of \( \Utheta(u) \le v \).

    Inductively, let \( 0 \le r < k \).
    Then, the canonical \( r \)\nbd pre-layering of \( \theta \) gives a decomposition of \( u \) into 
    \begin{equation*}
       \order{1}{u} \cp{r} \ldots \cp{r} \order{m}{u} 
    \end{equation*}
    for some \( m \geq 1 \) such that \( \order{i}{u} \) is an \( \ivl{r + 1}{k} \)\nbd theta.
    Since \( r < k \), we have by definition that 
    \begin{equation*}
        \Utheta(u) = \Utheta(\order{1}{u}) \cp{r} \ldots \cp{r} \Utheta(\order{m}{u}),
    \end{equation*}
    for all \( i \in \set{1, \ldots, m} \). 
    Let \( v \in \SdiS{\ivl{0}{k - 1}} U \) such that \( \Utheta(u) \le v \).
    Since \( \Utheta(u) \le v \) is active, the above \( r \)\nbd pre-layering of \( \Utheta(u) \) induces the \( r \)\nbd pre-layering \( (\order{i}{v})_{i = 1}^m \) of \( v \) such that \( \Utheta(\order{i}{u}) \le \order{i}{v} \).    
    Let \( i \in \set{1, \ldots, m} \).
    If \( \dim \order{i}{u} \le k \), then \( \order{i}{u} = \Utheta(\order{i}{u}) \le \order{i}{v} \) is already cocartesian in \( \SdiS{\ivl{0}{k}} U \).
    Otherwise, we may apply the inductive hypothesis to \( \Utheta(\order{i}{u}) \le \order{i}{v} \), since \( \pcell{\order{i}{u}} \colon \order{i}{U} \submol U \) is a frame-acyclic submolecule of \( U \) with \( \frdim \order{i}{U} \le k \le \dim \order{i}{U} - 1 \).
    Indeed, since \( r < k \), any maximal element of dimension \( > k \) in \( \order{i}{U} \) is also maximal in \( U \), but \( \frdim U \le k \) by assumption.
    We therefore obtain, for all \( i \in \set{1, \ldots, m} \), a cocartesian lift \( \order{i}{u} \le \order{i}{\hat{v}} \) of \( \Utheta(\order{i}{u}) \le \order{i}{v} \), and these can be pasted together into a lift 
    \begin{equation*}
        u \le \order{1}{\hat{v}} \cp{r} \ldots \cp{r} \order{m}{\hat{v}}
    \end{equation*}
    of \( \Utheta(u) \le v \), which is directly seen to be cocartesian. 
    This completes the induction and the proof.
\end{proof}

\begin{cor} \label{cor:compose_is_cocartesian}
    Let \( U \) be a frame-acyclic \( n \)\nbd molecule and let \( k \in \nat \) with \( \frdim U - 1 \le k \le \dim U - 1 \).
    Then the functor \(  \Utheta \colon \Sdi U \to \SdiS{\ivl{0}{k}} U \) is a cocartesian fibration.
\end{cor}
\begin{proof}
    Since cocartesian fibrations compose, we may repeatedly use Lemma \ref{lem:compose_is_cocartesian_one_step}.
\end{proof}

\begin{thm} \label{thm:poset_of_subdivision_contractible_or_empty}
    Let \( U \) be a frame-acyclic \( n \)\nbd molecule which is not an atom.
    Then \( \Sd U \) is contractible. 
\end{thm}
\begin{proof}
    Proceed as in the proof of \cite[Theorem 4.21]{campion2023pasting} with \( k \eqdef \frdim U + 1 \) (recall Comment \ref{comm:shift_index_preorder}), using Corollary \ref{cor:compose_is_cocartesian} and Proposition \ref{prop:frame_acyclic_theta_Sd_r_contractible}.
\end{proof}

\begin{comm} \label{comm:sd_u_not_contractible}
    The previous Theorem generally stops being true when the frame-acyclicity assumption is dropped.
    For instance, let \( U \) be the molecule of \cite[Example 8.2.20]{hadzihasanovic2024combinatorics}.
    This is a 4-dimensional molecule which is not frame-acyclic (and such that \( \molecin{U} \) is not a polygraph).
    Then \( U \) splits both as \( \order{1}{V} \cp{3} \order{2}{V} \) and \( \order{1}{W} \cp{3} \order{2}{W} \), yet there is no way to split \( U \) as \( A \cp{2} B \), with proper submolecules \( A \) and \( B \). 
    From there, it can easily be seen that the two \( 3 \)\nbd layerings \( (\order{1}{V}, \order{2}{V}) \) and \( (\order{1}{W}, \order{2}{W}) \) of \( U \) belong to two disconnected components of \( \Sd U \).
    In general, we do not know what are the possible homotopy types of \( \Sd U \), when \( U \) ranges over molecules.
\end{comm}
\section{The pasting theorem}
\label{sec:pasting}

\subsection{Recollection on \texorpdfstring{$(\infty, n)$}{(infty, n)}-categories}

\noindent If \( \C \) is a small category, we write \( \sPsh(\C) \) for the \( \infty \)\nbd category of space-valued presheaves on \( \C \), and recall that \( \Psh(\C) \), the usual \( 1 \)\nbd category of presheaves over \( \C \), is its full sub-\( \infty \)\nbd category on \( 0 \)\nbd truncated objects. 

In this section, we fix \( n \in \nat \cup \set{\omega} \).

\begin{dfn} [Fundamental spine] \label{dfn:spine}
    Let \( U \) be a regular directed \( n \)\nbd complex.
    The \emph{fundamental spine of \( U \)} is the presheaf over \( \Thetan{n} \)
    \begin{equation*}
        \Spine{U} \eqdef \colim_{x \in U} \molecin{\imel{U}{x}}.
    \end{equation*}  
    The \emph{fundamental spine inclusion of \( U \)} is the canonical morphism
    \begin{equation*}
        \spine U \colon \Spine{U} \to \molecin{U}.
    \end{equation*}
\end{dfn}

\begin{rmk} \label{rmk:spine_explicit}
    Explicitly, \( \Spine U \) is the subobject of \( \molecin{U} \) on those cells \( u \colon \theta \to \molecin{U} \) such that the local embedding \( \pcell{u} \) factors through an embedding of the form \( \mapel{x} \colon \imel{U}{x} \incl U \), for some \( x \in U \).
    It is also the left Kan extension of the functor \( U \mapsto \molecin{U} \) along the full subcategory inclusion \( \atomcat \incl \rdcpxequal \).
\end{rmk}

\begin{exm}
    When \( n = 1 \) and \( U = k\arr \), then \( \molecin{U} \) is the \( k \)\nbd simplex, and \( \spine{U} \) is the spine inclusion.
\end{exm}

\begin{exm}
    If \( U \) is an atom, then \( \spine{U} \) is an isomorphism.
\end{exm}

\begin{dfn} [Flagged \( (\infty, n) \)\nbd categories]
    The \( \infty \)\nbd category \( \nCatflag n \) of \emph{flagged \( (\infty, n) \)\nbd categories} is the localisation of \( \sPsh(\Thetan{n}) \) at the set \( \set{\spine{U}}_{U \in \Thetan{n}} \).
    We write
    \begin{equation*}
        \locflag \colon \sPsh(\Thetan{n}) \to \nCatflag n  
    \end{equation*}
    for the localisation functor, which has a right adjoint \( \N \).
\end{dfn}

\begin{rmk}
    The restriction of \( \N \) to \( 0 \)\nbd truncated objects gives a commutative square of fully-faithful right adjoints
    \begin{center}
        \begin{tikzcd}
            {\nCat n  } & {\nCatflag n} \\
            {\Psh(\Thetan n)} & {\sPsh(\Thetan n),}
            \arrow["\strweak", hook, from=1-1, to=1-2]
            \arrow["\N"', hook, from=1-1, to=2-1]
            \arrow["\N", hook, from=1-2, to=2-2]
            \arrow[hook, from=2-1, to=2-2]
        \end{tikzcd}
    \end{center}
    where \( \N \colon \nCat n \incl \Psh(\Thetan n) \) is the identification of strict \( n \)\nbd categories with presheaves over \( \Thetan n \) mentioned in (\ref{dfn:strict_n_cat_are_presheaf_theta}).
\end{rmk}

\begin{dfn} [Flagged equivalence]
    We say that a morphism in \( \sPsh(\Thetan{n}) \) is a \emph{flagged equivalence} if its image via \( \locflag \) is an equivalence.
\end{dfn}

\begin{dfn} [Flagged model structure]
    Let \( n \in \nat \cup \set{\omega} \).
    We write \( \dPsh(\Thetan{n}) \) for the category of simplicial presheaves on \( \Thetan{n} \).
    The injective model structure (which coincides with the Reedy model structure) on \( \dPsh(\Thetan{n}) \) has monomorphisms for cofibrations and presents \( \sPsh(\Thetan{n}) \).
    
    Now, \( \nCatflag{n} \) is presented by the \emph{flagged model structure}, which is the left Bousfield localisation of the injective model structure on \( \dPsh(\Thetan{n}) \) at the set of \( \set{\spine{U}}_{U \in \Thetan{n}} \), seen as discrete simplicial presheaves.
    A morphism in \( \nCatflag{n} \) is a flagged equivalence if it is presented by a weak equivalence in the flagged model structure.
\end{dfn}

\begin{rmk} \label{rmk:homotopy_pushout}
    The model structures of the previous definition are all Cisinski model structures.
    In particular, all objects are cofibrant, cofibrations are the monomorphisms, and any pushout along a cofibration is a homotopy pushout, see \cite[Section 2.3]{cisinski2019higher}.
    More generally, a commutative square 
    \begin{center}
        \begin{tikzcd}
            X & Z \\
            Y & W,
            \arrow[from=1-1, to=1-2]
            \arrow[hook, from=1-1, to=2-1]
            \arrow[from=1-2, to=2-2]
            \arrow[from=2-1, to=2-2]
        \end{tikzcd}
    \end{center}
    whose left leg is a monomorphism, is a homotopy pushout if the canonical morphism \( Y +_X Z \to W \) is a weak equivalence.
\end{rmk}

\begin{dfn} [Homotopy polygraph] \label{dfn:htpy_polygraph}
    Let \( C \) be an \( n \)\nbd polygraph.
    We say that \( C \) is a \emph{homotopy polygraph} if the cellular extensions of Definition \ref{dfn:polygraph} are homotopy pushouts in \( \nCatflag{n} \), that is, if they are preserved by the inclusion \( \strweak \colon \nCat{n} \incl \nCatflag{n} \).
\end{dfn}

\subsection{A pasting theorem for frame-acyclic directed complexes}

\noindent Our goal is to show that \( \spine{U} \) is a flagged equivalence as soon as \( U \) is a frame-acyclic molecule. 
The strategy of \cite[Section 5]{campion2023pasting} almost applies directly after Theorem \ref{thm:poset_of_subdivision_contractible_or_empty}.

\begin{dfn}
    Let \( U \) be a regular directed \( n \)\nbd complex.
    We define the following subobjects of \( \molecin{U} \):
    \begin{align*}
        \Smol U  &\eqdef \Spine U \cup \set{u \colon \theta \to \molecin{U} \mid u \text{ quotient-free}} \text{, and } \\
        \sSmol U &\eqdef \Spine U \cup \set{u \colon \theta \to \molecin{U} \mid u \text{ quotient-free}, \pcell{u} \neq \idd{U}}.
    \end{align*}
    If \( V \subseteq U \) is a closed subset, then there is a canonical inclusion \( \Smol V \subseteq \Smol U \), which factors through \( \sSmol U \), if \( V \) is a proper subset.
\end{dfn}

\begin{rmk}
    Since precomposing \( u \colon \theta \to \molecin{U} \) with an active strict functor does not change the principal pasting diagram, and since all local embeddings of thetas are embeddings by Proposition \ref{prop:ternary_factorisation_of_theta}, \( \Smol U \) and \( \sSmol U \) are well-defined subobjects of \( \molecin{U} \), which are equal if the regular directed complex \( U \) is not a molecule.
\end{rmk}

\begin{prop} \label{prop:frame_acyclic_can_add_active}
    Let \( U \) be a frame-acyclic \( n \)\nbd molecule.
    Then the inclusion \( \sSmol U \incl \Smol U \) is a flagged equivalence.
\end{prop}
\begin{proof}
    We work in the model structure on \( \dPsh(\Thetan{n}) \) presenting \( \nCatflag{n} \).  
    As in \cite[Definition 5.6]{campion2023pasting}, for any regular directed complex \( V \), we let \( \Smolp V \) be the simplicial presheaf over \( \Thetan{n} \) defined by sending \( \xi \) to the nerve of the category whose objects are pairs 
    \begin{equation*}
        (\F \colon \xi \to \theta, u \colon \theta \to \Smol V)
    \end{equation*}
    such that \( u \) is non-degenerate (hence a monomorphism by Proposition \ref{prop:quotient-free_non_degenerate_are_mono}) and \( u \neq \bigcell{V} \), and a morphism 
    \begin{equation*}
         \G \colon (\F \colon \xi \to \theta, u) \to (\F' \colon \xi \to \theta', u')
    \end{equation*}
    is given by a strict functor \( \G \colon \theta \to \theta' \) making the two triangles commute.
    There is a projection morphism \( p \colon \Smolp V \to \Smol V \) sending \( (\F, u) \) to \( u \after \F \).
    We then let \( \sSmolp U \) be the pushout
    \begin{center}
        \begin{tikzcd}
            \begin{array}{c} \bigcup_{\substack{V \subsetneq U}} \Smolp V \end{array} & {\sSmol U} \\
            {\Smolp U} & {\sSmolp U,}
            \arrow[""{name=0, anchor=center, inner sep=0}, from=1-1, to=1-2]
            \arrow[hook, from=1-1, to=2-1]
            \arrow[hook, from=1-2, to=2-2]
            \arrow[from=2-1, to=2-2]
            \arrow["\lrcorner"{anchor=center, pos=0.125, rotate=180}, draw=none, from=2-2, to=0]
        \end{tikzcd}
    \end{center}
    where the \( V \) run over the proper closed subsets of \( U \).
    This comes equipped, using the universal property, with a projection morphism \( \gamma \colon \sSmolp U \to \sSmol U \). 
    By the proof of \cite[Proposition 5.7]{campion2023pasting} together with Theorem \ref{thm:poset_of_subdivision_contractible_or_empty}, \( \gamma \) is a flagged equivalence.
    Last, the inclusion \( \sSmol U \incl \sSmolp U \) is a flagged equivalence by the proof of \cite[Lemma 5.8]{campion2023pasting} and recalling Proposition \ref{prop:quotient-free_non_degenerate_are_mono}.
    Since postcomposing this inclusion with \( \gamma \) is the inclusion \( \sSmol U \incl \Smol U \), we conclude.
\end{proof}

\begin{dfn} [Directed complex with frame-acyclic molecules]
    Let \( X \) be a directed complex.
    We say that \emph{\( X \) has frame-acyclic molecules} if, for all pasting diagrams \( f \colon U \to X \), \( U \) is frame-acyclic.
\end{dfn}

\noindent We may now state and prove our analogue of \cite[Theorem 5.11]{campion2023pasting}.

\begin{thm} \label{thm:cellular_extension_of_frame_acyclic}
    Let \( X \) be a directed \( n \)\nbd complex with frame-acyclic molecules and \( d \in \nat \).
    Then the cellular extension
    \begin{center}
        \begin{tikzcd}[column sep=large]
            {\coprod_{x \in \gr{d} X} \molecin{\bd{}{} U}} & {\coprod_{x \in \gr{d} X} \molecin{U}} \\
            {\molecin{\skel{d - 1}{X}}} & {\molecin{\skel{d}{X}}}
            \arrow[hook, from=1-1, to=1-2]
            \arrow[from=1-1, to=2-1]
            \arrow[from=1-2, to=2-2]
            \arrow[hook, from=2-1, to=2-2]
        \end{tikzcd}
    \end{center}
    is a homotopy pushout in \( \nCatflag{n} \).
\end{thm}
\begin{proof}
    We work in the model structure over \( \dPsh(\Thetan{n}) \) presenting \( \nCatflag{n} \). 
    Let \( K \) be the cellular extension computed in \( \dPsh(\Thetan{n}) \).
    By Remark \ref{rmk:homotopy_pushout}, we must show that the canonical morphism \( \iota \colon K \to \molecin{\skel{d}{X}} \) is a flagged equivalence.
    By Corollary \ref{cor:strict_functor_out_of_theta}, \( \iota \) is an inclusion classifying the subobject of \( \molecin{\skel{d}{X}} \) on those \( u \colon \theta \to \molecin{\skel{d}{X}} \) such that \( \dim \pcell{u} < d \) or \( \pcell{u} \) is a cell of \( X \).
    For each pasting diagram \( f \colon U \to \skel{d}{X} \), we write 
    \begin{align*}
        K_f &\eqdef K \cup \set{u \colon \theta \to \molecin{\skel{d}{X}} \mid \pcell{u} \subseteq f} \text{, and} \\
        K_{<f} &\eqdef K \cup \set{u \colon \theta \to \molecin{\skel{d}{X}} \mid \pcell{u} \subseteq f, \pcell{u} \neq f},
    \end{align*}
    whereby \( \pcell{u} \subseteq f \), we mean that there exists an embedding \( \iota \colon V \incl U \) such that \( f \after \iota = \pcell{u} \).
    Notice that we have an inclusion \( \iota_f \colon K \incl K_f \), which is an equality whenever \( \dim f < d \) or \( f \) is a cell.
    We write \( \mathcal{S} \) for the collection of all inclusions \( \iota_f \), for \( f \) a pasting diagram in \( \skel{d}{X} \).
    We wish to apply \cite[Lemma 2.18]{campion2023pasting} to \( \mathcal{S} \). 
    This will be enough to conclude that \( \iota \) is flagged equivalence, as we have
    \begin{equation*}
        \molecin{\skel{d}{X}} = \bigcup_{f \in \mathcal{S}} K_f,
    \end{equation*}
    since for any \( u \colon \theta \to \molecin{\skel{d}{X}} \), \( u \in \fun{K}_{\pcell{u}}(X) \). 
    Now, for each pair \( \iota_f, \iota_{f'} \in \mathcal{S} \),
    \begin{equation*}
         K_f \cap K_{f'} = \bigcup_{\substack{\iota_g \in \mathcal{S}, \\ g \subseteq f, g \subseteq f'}} K_g,
    \end{equation*}
    hence it only remains to show that the elements of \( \mathcal{S} \) are all flagged equivalences.

    Consider a pasting diagram \( f \colon U \to \skel{d}{X} \), we show that \( \iota_f \) is a flagged equivalence by induction on the number of elements in the underlying set of \( U \).
    For the base case, then \( U \) has one element and is the point \( \pt \).
    Therefore, \( \iota_f \) is an isomorphism, since \( f \) is a cell.
    Inductively, suppose that the statement holds for all pasting diagrams \( g \colon V \to \skel{d}{X} \), where \( V \) has fewer elements than \( U \).
    If \( U \) is an atom, then again \( \iota_{f} \) is an isomorphism since \( f \) is a cell.
    Thus, we may suppose that \( U \) is a molecule.
    In that case, we may factor \( \iota_f \) as
    \begin{equation*}
        K \incl K_{<f} \incl K_f.
    \end{equation*}
    We claim that both inclusions are flagged equivalences.
    For the first one, we apply \cite[Lemma 2.18]{campion2023pasting} to the subset \( \mathcal{S}_{<f} \) of \( \mathcal{S} \) consisting of those \( \iota_g \in \mathcal{S} \) such that \( g \subseteq f \) and \( g \neq f \).
    Indeed,
    \begin{equation*}
        K_{<f} = \bigcup_{\iota_g \in \mathcal{S}_{< f}} K_g,
    \end{equation*}
    and, for all \( \iota_g \in \mathcal{S}_{< f} \), \( \iota_{g} \) is by definition of the form \( \iota_{f \after j} \) with \( j \colon V \incl U \) a proper embedding, thus the inductive hypothesis applies and shows that \( \iota_g \) is a flagged equivalence. 
    For the second one, we claim that the diagram
    \begin{center}
        \begin{tikzcd}
            {\sSmol U} & {\Smol U} \\
            {K_{<f}} & {K_f}
            \arrow[hook, from=1-1, to=1-2]
            \arrow[from=1-1, to=2-1]
            \arrow[from=1-2, to=2-2]
            \arrow[hook, from=2-1, to=2-2]
        \end{tikzcd}
    \end{center}
    is a pushout square.
    Since weak equivalences in the flagged model structure on \( \dPsh(\Thetan{n}) \) are stable by cobase change, Proposition \ref{prop:frame_acyclic_can_add_active} will be enough to conclude.
    To see that the previous diagram is a pushout square, let \( u \colon \theta \to \molecin{\skel{d}{X}} \) be a non-degenerate cell in \(K_f \) which is not in \( K_{<f} \).
    By Corollary \ref{cor:strict_functor_out_of_theta}, \( u \) is of the form \( \molecin{\pcell{u}} \after \hat{u} \), where \( \hat{u} \) is active. 
    Since \( u \in K_f \) but \( u \notin K_{<f} \), we have that \( \pcell{u} = f \), hence \( u \) is the image of \( \hat{u} \in \Smol U \).
    Finally, by Corollary \ref{cor:ofs_functor_of_rdcpx}, \( \molecin{f} \after - \) is injective on active strict functors.
    This concludes the proof.
\end{proof}

\begin{comm}
    The reader may be surprised to find, in the previous proof, an induction on the number of elements of a molecule, instead of the more classical induction on submolecules.
    Our reason for this choice is that an induction on the submolecules would have required to define \( \Smol U \) (and \( K_f \)) not as classifying all quotient-free strict functors, but merely the one such that \( \pcell{u} \) is a subdiagram of \( \idd{U} \) (or \( f \)).
    This is not a problem, except that affirming this gives a well-defined subobject requires knowing that subdiagrams are stable when restricted along embeddings of the form \( \iota \colon \theta' \incl \theta \), a fact that only follows from knowing that any embedding of thetas is a submolecule inclusion.
    While we have no doubt that this is true, we decided not to prove it in this article.
\end{comm}

\noindent We may now extend \cite[Theorem 8.2.14]{hadzihasanovic2024combinatorics}: directed complexes with frame-acyclic molecules are polygraphs and homotopy polygraphs.

\begin{thm} \label{thm:frame_acyclic_is_homotopy_polygraph}
    Let \( X \) be a directed \( n \)\nbd complex with frame-acyclic molecules.
    Then the strict \( n \)\nbd category \( \molecin{X} \) is a polygraph and a homotopy polygraph.
\end{thm}
\begin{proof}
    Since the left adjoint \( \pi_0 \colon \nCatflag{n} \to \nCat{n} \) sends homotopy pushouts to pushouts in strict \( n \)\nbd categories, we conclude by Theorem \ref{thm:cellular_extension_of_frame_acyclic} first that it is a polygraph, and second that it is a homotopy polygraph. 
\end{proof}

\begin{cor} \label{cor:spine_of_frame_acyclic_rdcpx_is_flagged_equivalence}
    Let \( U \) be a regular directed \( n \)\nbd complex with frame-acyclic molecules.
    Then \( \spine{U} \) is a flagged equivalence. 
\end{cor}
\begin{proof}
    We work in the model structure on \( \dPsh(\Thetan{n}) \) presenting \( \nCatflag{n} \) and prove the statement by induction on \( \dim U \).
    When \( \dim U = 0 \), then \( \spine{U} \) is in fact an isomorphism.
    Inductively, suppose that \( d \eqdef \dim U > 0 \), then since \( \Spine \) commutes with colimits of embeddings of regular directed complexes, the inductive hypothesis, Theorem \ref{thm:cellular_extension_of_frame_acyclic}, Remark \ref{rmk:spine_explicit}, and Remark \ref{rmk:homotopy_pushout} imply that the two squares
    \begin{center}
        \begin{tikzcd}
            {\coprod\limits_{x \in \gr{d} U} \molecin{\bd{}{} V}} & {\coprod\limits_{x \in \gr{d} U} \molecin{V}} & {\coprod\limits_{x \in \gr{d} U} \Spine{\bd{}{} V}} & {\coprod\limits_{x \in \gr{d} U} \Spine V} \\
            {\molecin{\skel{d - 1}{U}}} & {\molecin{\skel{d}{U}}} & {\Spine \skel{d - 1}{U}} & {\Spine \skel{d}{U}}
            \arrow[from=1-1, to=1-2]
            \arrow[from=1-1, to=2-1]
            \arrow[from=1-2, to=2-2]
            \arrow[""{name=0, anchor=center, inner sep=0}, hook, from=1-3, to=1-4]
            \arrow[from=1-3, to=2-3]
            \arrow[from=1-4, to=2-4]
            \arrow[from=2-1, to=2-2]
            \arrow[hook, from=2-3, to=2-4]
            \arrow["\lrcorner"{anchor=center, pos=0.125, rotate=180}, draw=none, from=2-4, to=0]
        \end{tikzcd}
    \end{center}
    present the same homotopy pushout, so \( \spine{\skel{d}{U}} \) is a flagged equivalence.
    This concludes the induction and the proof.
\end{proof}

\begin{cor} \label{cor:spine_of_3_rdcpx_is_flagged_equivalence}
    Let \( U \) be a regular directed \( 3 \)\nbd complex.
    Then \( \spine{U} \) is a flagged equivalence.
\end{cor}
\begin{proof}
    By Corollary \ref{cor:spine_of_frame_acyclic_rdcpx_is_flagged_equivalence} and \cite[Theorem 8.4.11]{hadzihasanovic2024combinatorics}.  
\end{proof}

\begin{comm}
    The results of this section stop being generally true when the frame-acyclic assumption is dropped: as in Comment \ref{comm:sd_u_not_contractible}, consider the non-frame-acyclic molecule \( U \) of \cite[Example 8.2.20]{hadzihasanovic2024combinatorics}, then \( \molecin{U} \) is not a polygraph.
    We claim that \( \spine{U} \) is not a flagged equivalence. 
    Indeed, let \( C \eqdef \colim_{x \in U} \molecin{\imel{U}{x}} \), where the colimit is taken in strict \( \omega \)\nbd categories. 
    Then, \( C \) is a gaunt polygraph and there is a canonical morphism \( u \colon \Spine U \to C \).
    Since \( \rs C \cong \molecin{U} \), the unit of the reflector \( \rs \) gives a canonical strict functor \( v \colon C \to \molecin{U} \). 
    If \( \spine{U} \) were an equivalence, then \( u \) would extend to a strict functor \( \hat{u} \colon \molecin{U} \to C \), and it is straightforward to check that \( \hat{u} \) and \( v \) are inverses of each other.
    Therefore, \( \molecin{U} \cong \rs C \cong C \) is a polygraph, a contradiction.
\end{comm}

\subsection{Applications}

\noindent We conclude by demonstrating the range of applications of Theorem \ref{thm:cellular_extension_of_frame_acyclic} by giving a large supply of directed complexes with frame-acyclic molecules, including regular \( 3 \)\nbd polygraphs (Corollary \ref{cor:3regular_polygraph_are_homotopy_polygraph}).
  
\begin{dfn} [Oriented Hasse diagram]
    Let \( U \) be an oriented graded poset.
    The \emph{oriented Hasse diagram of \( U \)} is the directed graph whose:
    \begin{itemize}
        \item set of vertices is the underlying set of \( U \), and
        \item for each pair of vertices \( x, y \), there is an edge \( x \to y \) if \( x \in \faces{}{-} y \) or \( y \in \faces{}{+} x \).
    \end{itemize}
\end{dfn}

\begin{dfn} [Acyclic molecule]
    Let \( U \) be a molecule.
    We say that \( U \) is \emph{acyclic} if its oriented Hasse diagram is acyclic.
    We write \( \atomcatac \) for the full subcategory of \( \atomcat \) on acyclic atoms.
\end{dfn}

\begin{dfn} [Directed complex with acyclic atoms] \label{dfn:acyclic_atoms}
    A \emph{directed complex with acyclic atoms} is a presheaf over the category \( \atomcatac \).
    We write \( \dcpxac \) for the category of directed complexes with acyclic atoms and their natural transformations.
\end{dfn}

\begin{rmk}
    The category \( \dcpxac \) can be identified with the full subcategory of \( \dcpx \) on those directed complexes \( X \) with the property that for all cells \( x \colon U \to X \), the atom \( U \) is acyclic. 
\end{rmk}

\begin{rmk}
    It is straightforward to check that the class of directed complexes with acyclic atoms is a good class of polygraphs in the sense of \cite[Definition 2.2.4]{henry2018nonunital}.
    It is, however, not algebraic in the sense of \cite[Definition 3.1.1]{henry2018regular}, since the input and output boundaries of an atom may both be acyclic without this atom being itself acyclic, for otherwise all molecules would be acyclic. 
\end{rmk}

\begin{lem} \label{lem:acyclic_dcpx_is_frame_acyclic}
    Let \( X \) be a directed complex with acyclic atoms.
    Then \( X \) has frame-acyclic molecules.
\end{lem}
\begin{proof}
    Let \( f \colon U \to X \) be a pasting diagram in \( X \).
    By induction on the submolecules of \( U \) and \cite[Lemma 8.3.26]{hadzihasanovic2024combinatorics}, \( U \) is an acyclic molecule.
    By \cite[Proposition 8.3.6 and Proposition 8.3.11]{hadzihasanovic2024combinatorics}, \( U \) is frame-acyclic.
\end{proof}

\begin{cor} \label{cor:dcpx_with_acyclic_atom_is_homotopy_polygraph}
    Let \( X \) be a directed complex with acyclic atoms.
    Then \( \molecin{X} \) is a polygraph and a homotopy polygraph.
\end{cor}
\begin{proof}
    By Lemma \ref{lem:acyclic_dcpx_is_frame_acyclic} and Theorem \ref{thm:frame_acyclic_is_homotopy_polygraph}.
\end{proof}

\begin{prop} \label{prop:acyclic_into_flagged_preserves_colimits}
    The functor \( \strweak\molecin{-} \colon \dcpxac \to \nCatflag{\omega} \) sends colimits to homotopy colimits.
\end{prop}
\begin{proof}
    By Corollary \ref{cor:dcpx_with_acyclic_atom_is_homotopy_polygraph}, \( \strweak\molecin{-} \) is the left Kan extension of its restriction along the Yoneda embedding \( \atomcatac \incl \dcpxac \).
    This concludes the proof. 
\end{proof}

\noindent Following \cite[Section 7.2, Corollary 8.3.34]{hadzihasanovic2024combinatorics}, we write \( - \gray - \) for the Gray product of regular directed complexes, which restricts to a monoidal structure \( (\atomcatac, \gray, \pt) \) on the category of acyclic atoms.
By Day convolution \cite{day1970closed}, we may extend the Gray product to a monoidal structure \( (\dcpxac, \gray, \pt) \) on the category of directed complexes with acyclic atoms.
We recall that the \( \infty \)\nbd category \( \nCatflag{\omega} \) also possesses a monoidal structure \( (\nCatflag{\omega}, \gray, \pt) \) given by the Gray product, see \cite{campion2023gray} for a detailed treatment. 

\begin{thm} \label{thm:gray_monoidal_acyclic}
    The functor \( \strweak\molecin{-} \colon \dcpxac \to \nCatflag{\omega} \) lifts to a (strong) monoidal functor
    \begin{equation*}
        \strweak\molecin{-} \colon (\dcpxac, \gray, \pt) \to (\nCatflag{\omega}, \gray, \pt).
    \end{equation*}
\end{thm}
\begin{proof}
    By Proposition \ref{prop:acyclic_into_flagged_preserves_colimits} and the universal property of the Day convolution \cite{kelly1986day}, it is enough to show that the restriction of \( \strweak\molecin{-} \) along the Yoneda embedding \( \atomcatac \incl \dcpxac \) is monoidal.
    By \cite[Theorem 11.2.18, Proposition 11.2.36]{hadzihasanovic2024combinatorics}, the functor \( \molecin{-} \colon \atomcatac \to \omegaCat \) essentially lands in strict \( \omega \)\nbd categories free on a strongly loop-free Steiner complex \cite{steiner2004omega}, and is monoidal for the Gray tensor product.
    Using \cite[Example 3.2 and Corollary 3.5]{campion2023gray}, we conclude that \( \strweak\molecin{-} \colon \atomcatac \to \nCatflag{\omega} \) is monoidal.
\end{proof}

\begin{comm}
    A similar result holds of the join of directed complexes with acyclic atoms.
\end{comm}

\noindent Next, we give an important class of directed complexes with acyclic atoms: the semi-simplicial sets. 
We write \( - \join - \) for the join of regular directed complexes, see \cite[Section 7.4]{hadzihasanovic2024combinatorics} for more details.

\begin{dfn} [Oriented simplex]
    Let \( k \in \nat \).
    The \emph{oriented \( k \)\nbd simplex} is the atom
    \begin{equation*}
        \ospx k \eqdef \underbrace{\pt \join \ldots \join \pt}_{(k + 1) \text{ times}}.
    \end{equation*}
\end{dfn}

\begin{exm}
    Let \( k \in \nat \).
    Then \( \molecin{\ospx{k}} \) is the \( k \)\nbd th oriental \cite{street1987algebra}.
\end{exm}

\begin{dfn} [\( \ospx{} \)\nbd complex]
    A \emph{\( \ospx{} \)\nbd complex} is a directed complex \( X \) with the property that for all cells \( x \colon U \to X \), the atom \( U \) is an oriented simplex.
\end{dfn}

\begin{rmk} \label{rmk:semi_simplicial_cubical_site}
    The result of \cite[Proposition 9.2.14]{hadzihasanovic2024combinatorics} implies that the full subcategory of \( \atomcat \) on oriented simplices is isomorphic to the semi-simplex category (finite linear orders and order-preserving injections).
    By generalities on left Kan extensions, the category of semi-simplicial sets\footnote{Also called \( \Delta \)\nbd complexes.} coincides with the full subcategory of \( \dcpx \) on \( \ospx{} \)\nbd complexes.
\end{rmk}

\begin{thm} \label{thm:simplicial_homotopy_polygraph}
    Let \( X \) be a \( \ospx{} \)\nbd complex.
    Then \( \molecin{X} \) is a polygraph and a homotopy polygraph.
\end{thm}
\begin{proof}
    By \cite[Lemma 9.2.3]{hadzihasanovic2024combinatorics}, oriented simplices are acyclic.
    We conclude by Corollary \ref{cor:dcpx_with_acyclic_atom_is_homotopy_polygraph}.
\end{proof}

\noindent Finally, we conclude the article with a few corollaries in low dimension, where frame-acyclicity conditions are automatically true.
The following is well-known, see also \cite[\href{https://kerodon.net/tag/00J6}{Tag 00J6}]{kerodon}.

\begin{cor} \label{cor:free_categories_on_graph}
    The free category on a (directed) graph is also its free \( \infty \)\nbd category.
\end{cor}
\begin{proof}
    A directed graph is equivalently a directed \( 1 \)\nbd complex, which is also a \( \ospx{} \)\nbd complex, since the only atoms of dimension \( \le 1 \) are directed simplices.
    Since the free category on a graph \( G \) is given by \( \molecin{G} \), we conclude by Theorem \ref{thm:simplicial_homotopy_polygraph}.
\end{proof}

\begin{thm} \label{thm:dim_le_3_frame_acyclic}
    Let \( X \) be a directed \( 3 \)\nbd complex.
    Then \( \molecin{X} \) is a polygraph and a homotopy polygraph.
\end{thm}
\begin{proof}
    Let \( f \colon U \to X \) be a pasting diagram.
    Then \( \dim U \le 3 \) and by \cite[Theorem 8.4.11]{hadzihasanovic2024combinatorics}, \( U \) is frame-acyclic.
    Thus \( X \) has frame-acyclic molecules and we conclude by Theorem \ref{thm:frame_acyclic_is_homotopy_polygraph}.
\end{proof}

\begin{cor} \label{cor:3regular_polygraph_are_homotopy_polygraph}
    Let \( C \) be a regular \( 3 \)\nbd polygraph.
    Then \( C \) is a homotopy polygraph.
\end{cor}
\begin{proof}
    By Theorem \ref{thm:from_polyplexes_to_molecules} and Theorem \ref{thm:dim_le_3_frame_acyclic}.
\end{proof}

\begin{cor} \label{cor:2positive_are_homotopy}
    Let \( C \) be a positive \( 2 \)\nbd polygraph.
    Then \( C \) is a homotopy polygraph.
\end{cor}
\begin{proof}
    Any positive \( 2 \)\nbd polygraph is regular.
    We conclude by the previous Corollary.
\end{proof}

\begin{comm}
    It is believed that Corollary \ref{cor:3regular_polygraph_are_homotopy_polygraph} holds of any regular (and even positive) polygraph.
    As the results of this section demonstrate, Campion's proof should suffice once one knows that \( \Sd \pl{u} \) is contractible when \( \pl{u} \) is a regular polyplex which is not a plex.
    However, it will not be possible to deduce this fact purely from combinatorial arguments involving the underlying poset of a polyplex, since the latter is that of a molecule by Theorem \ref{thm:from_polyplexes_to_molecules}, and is not contractible in general, as per Comment \ref{comm:sd_u_not_contractible}.
\end{comm}

\bibliographystyle{alpha}
\small\bibliography{main.bib}

\end{document}